\documentclass[abstract=on]{scrartcl}
\setkomafont{sectioning}{\normalfont\normalcolor\bfseries}
\usepackage{amsmath}
\usepackage{amssymb}
\usepackage{amsthm}
\usepackage{mathtools}
\usepackage[T1]{fontenc}
\usepackage[utf8]{inputenc}
\usepackage{setspace}
\usepackage{graphicx}
\usepackage{subfig}
\usepackage{color}
\usepackage{cite}
\usepackage{lscape}
\usepackage{pdfpages}
\usepackage{caption}
\usepackage{bbm} %Zahlen mit Doppelstrich
\usepackage{enumitem} %Geänderte Aufzählungszeichen
\usepackage[usenames,dvipsnames]{xcolor}
\usepackage[color=SeaGreen]{todonotes} %\usepackage[disable]{todonotes}
\usepackage{comment}
\usepackage{mathrsfs}
\usepackage{algorithm}
\usepackage{algpseudocode}
\usepackage{bbm}
\usepackage{multirow}
\usepackage{longtable}
\usepackage{dcolumn}
\usepackage{booktabs}
\usepackage{subfig}
\usepackage{mathdots}
\usepackage{hhline}
\usepackage[space]{grffile}
\usepackage[numbered,framed]{matlab-prettifier}
\usepackage[hyperfootnotes=false,hidelinks]{hyperref}
\hypersetup{colorlinks=false}
\setenumerate{label=(\roman*)}

\usepackage{geometry}
\newtheoremstyle{mythm}
   {10pt}                   %Space above
   {10pt}                   %Space below
   {}          		    %Body font: original {\normalfont}
   {}                      %Indent amount
   {\normalfont\bfseries}  %Thm head font original {\normalfont\bfseries}
   {}                      %Punctuation after thm head original :
   {\newline}              %Space after thm head: " " = normal interword
   {\textbf{\thmname{#1} \thmnumber{#2} \thmnote{(#3)}}} %Thm head spec (can be left empty, meaning `normal') 

\newtheoremstyle{mythm2}
   {10pt}                   %Space above
   {10pt}                   %Space below
   {\itshape}          		    %Body font: original {\normalfont}
   {}                      %Indent amount
   {\normalfont\bfseries}  %Thm head font original {\normalfont\bfseries}
   {}                      %Punctuation after thm head original :
   {\newline}              %Space after thm head: " " = normal interword
   {\textbf{\thmname{#1} \thmnumber{#2} \thmnote{(#3)}}} %Thm head spec (can be left empty, meaning `normal') 

\newtheoremstyle{myex}
  {10pt}                   %Space above
   {10pt}                   %Space below
   {}          		    %Body font: original {\normalfont}
   {}                      %Indent amount
   {\normalfont\bfseries}  %Thm head font original {\normalfont\bfseries}
   {}                      %Punctuation after thm head original :		  
   {\newline}                %Space after thm head: " " = normal interword       %space; \newline = linebreak
   {\textbf{\thmname{#1} \thmnumber{#2} \thmnote{(#3)}}} %Thm head spec (can be left empty, meaning `normal')

\theoremstyle{myex}
\newtheorem*{XxmpX}{Example} % \newtheorem establishes the object heading
\newenvironment{ex}    % this is the environment name for the input
  {%
   \pushQED{\qed}\begin{XxmpX}}
  {\popQED\end{XxmpX}}
   
\theoremstyle{mythm2}
\newtheorem{thm}{Theorem}[section]
\theoremstyle{mythm}
\newtheorem{lem}[thm]{Lemma}
\newtheorem{dfn}[thm]{Definition}
\newtheorem{cor}[thm]{Corollary}

\newtheorem{Rem}[thm]{Remark} % \newtheorem establishes the object heading
\newenvironment{rem}    % this is the environment name for the input
  {%
   \pushQED{\qed}\begin{Rem}}
  {\popQED\end{Rem}}
  
\makeatletter
\newcommand{\subalign}[1]{%
  \vcenter{%
    \Let@ \restore@math@cr \default@tag
    \baselineskip\fontdimen10 \scriptfont\tw@
    \advance\baselineskip\fontdimen12 \scriptfont\tw@
    \lineskip\thr@@\fontdimen8 \scriptfont\thr@@
    \lineskiplimit\lineskip
    \ialign{\hfil$\m@th\scriptstyle##$&$\m@th\scriptstyle{}##$\crcr
      #1\crcr
    }%
  }
}
\makeatother

\DeclareMathOperator{\im}{Im}
\DeclareMathOperator{\re}{Re}
\DeclareMathOperator{\argmin}{argmin}
\DeclareMathOperator{\median}{median}
\DeclareMathOperator{\lcm}{lcm}
\newcommand{\eps}{\varepsilon}
\newcommand{\intvl}{\left(-\lceil N/2\rceil, \lfloor N/2\rfloor\right]\cap\mathbb{Z}}
\newcommand{\Intvl}{\left(-\left\lceil \frac{N}{2}\right\rceil, \left\lfloor \frac{N}{2}\right\rfloor\right]\cap\mathbb{Z}}
\newcommand{\intu}{-\left\lceil\frac{N}{2}\right\rceil+1}

\newcommand{\into}{\left\lfloor\frac{N}{2}\right\rfloor}
\newcommand{\mods}{\,\bmod\,}

\newcommand{\CC}{\mathbb{C}}
\newcommand{\NN}{\mathbb{N}}
\newcommand{\ZZ}{\mathbb{Z}}
\newcommand{\RRR}{\mathcal{R}_{L,K}}
\newcommand{\GGG}{\mathcal{G}_{L,K}}
\newcommand{\EEE}{\mathcal{E}_{K}}
\newcommand{\Ms}{\mathcal{M}_{s_1,K}}
\newcommand{\Mu}{\mathcal{M}_{u_1,M}}
\newcommand{\Nt}{\mathcal{N}_{t_1,L}}
\newcommand{\Mt}{\mathcal{M}_{t_1,L}}
\newcommand{\cast}{\circledast}
\newcommand{\x}{\mathbf{x}}

\newcommand{\F}{\mathbf{F}}
\newcommand{\f}{\mathbf{f}}
\newcommand{\cc}{\mathbf{c}}
\newcommand{\opt}{^{\mathrm{opt}}}
\newcommand{\mn}{^{(m,\nu)}}
\newcommand{\A}{\mathbf{A}}
\newcommand{\B}{\mathbf{B}}
\newcommand{\HHH}{\mathcal{H}_{M}}
\newcommand{\bfrho}{\boldsymbol{\rho}}
\newcommand{\rr}{\mathbf{r}}

\newcommand{\BIGOP}[1]{\mathop{\mathchoice%
{\raise-0.22em\hbox{\huge $#1$}}%
{\raise-0.05em\hbox{\Large $#1$}}{\hbox{\large $#1$}}{#1}}}

\begin{document}
%\raggedbottom
\title{A Deterministic Sparse FFT for Functions with Structured Fourier Sparsity}
\author{Sina Bittens\thanks{University of Göttingen, Institute for Numerical and Applied Mathematics, Lotzestr.\ 16-18, 37083 Göttingen, Germany (sina.bittens@mathematik.uni-goettingen.de, +49 551 394515).}, Ruochuan Zhang\thanks{Department of Mathematics, Michigan State University, East Lansing, MI, 48824, USA (zhangr12@msu.edu).}, Mark A. Iwen\thanks{Department of Mathematics, and Department of Computational Mathematics, Science, and Engineering (CMSE), Michigan State University, East Lansing, MI, 48824, USA (markiwen@math.msu.edu).}}
\date{\today}
\maketitle

\includecomment{paper}
\excludecomment{diss}
\begin{abstract}
In this paper a deterministic sparse Fourier transform algorithm is presented which breaks the quadratic-in-sparsity runtime bottleneck for a large class of periodic functions exhibiting structured frequency support.  These functions include, e.g., the oft-considered set of block frequency sparse functions of the form
$$f(x) = \sum^{n}_{j=1} \sum^{B-1}_{k=0} c_{\omega_j + k} e^{i(\omega_j + k)x},~~\{ \omega_1, \dots, \omega_n \} \subset \Intvl$$
as a simple subclass.  Theoretical error bounds in combination with numerical experiments demonstrate that the newly proposed algorithms are both fast and robust to noise.  In particular, they outperform standard sparse Fourier transforms in the rapid recovery of block frequency sparse functions of the type above.
\end{abstract}
\section{Introduction}\label{ch:intro}
In this paper we consider the problem of deterministically recovering a periodic function $f\colon[0, 2\pi] \rightarrow \mathbb{C}$ as rapidly as absolutely possible via sampling.  In particular, we focus on a specific set of functions $f$ whose dominant Fourier series coefficients are all associated with frequencies contained in a small number, $n$, of unknown structured support sets $S_1, \dots, S_n \subset \left( - \lceil N/2 \rceil, \lfloor N/2 \rfloor \right] \cap \mathbb{Z}$, where $N \in \mathbb{N}$ is very large.  In such cases the function $f$ will have the form
\begin{equation}
f(x) \approx \sum^n_{j=1} \sum_{\omega \in S_j} c_\omega e^{i \omega x},
\label{equ:Example_Problem}
\end{equation}
where each unknown $S_j$ has simplifying structure (e.g., has $|x - y| < B \ll N$ for all $x,y \in S_j$). 

The classical solution for this problem would be to compute the Discrete Fourier Transform (DFT) of $N$ equally spaced samples from $f$ on $[0, 2 \pi]$, $\left( f \left( \frac{2 \pi j}{N} \right) \right)^{N-1}_{j = 0}$, in order to obtain approximations of $c_\omega$ for all $\omega \in \left( - \lceil N/2 \rceil, \lfloor N/2 \rfloor \right] \cap \mathbb{Z}$ in $\mathcal{O}(N \log N)$-time.  Herein, we instead consider faster deterministic Sparse Fourier Transform (SFT) methods which are guaranteed to recover such $f$ using a number of samples and operations that scale at most polynomially in both $\sum^n_{j=1} |S_j|$ and $\log(N)$.  Such algorithms will always be faster than classical $\mathcal{O}(N \log (N))$-time methods whenever the cardinalities of the support sets, $|S_j|$, are sufficiently small in comparison to $N$.  The main contribution of this paper is the development of the fastest known deterministic SFT methods to date for the recovery of a large class of periodic functions of type \eqref{equ:Example_Problem}.  Such functions \eqref{equ:Example_Problem} will be referred to as functions with \textit{structured frequency support} below.

\subsection{Related Work:  Sparse Fourier Transforms}

The vast majority of the work on sparse Fourier transform methods has focused on the unstructured frequency sparse case where, e.g., each set $S_j$ in \eqref{equ:Example_Problem} is just a singleton set.  In this case functions of the form \eqref{equ:Example_Problem} are simply $n$-sparse in the Fourier domain.  The first sub-linear time methods developed for rapidly computing the Fourier series coefficients of such frequency sparse functions were randomized algorithms \cite{Man,GGIMS,AGS,GMS} which fail to output good solutions with some constant (and usually tunable) probability.  In exchange for this slight unreliability in producing accurate output, the fastest of these randomized techniques are able to compute the Fourier series of $n$-sparse $f$ in just $n \log^{\mathcal{O}(1)} N$-time.  The most efficient, numerically stable, and publicly available implementations of these methods are based on random algorithms developed out of MIT \cite{HIKP,HIKP2,IKP}, Michigan \cite{GGIMS,GMS,IGS}, and Michigan State  \cite{segal2013improved,christlieb2016multiscale,Ruochuan}.\footnote{The code for all of these implementations is freely available on the web \cite{sFFT,AAFFT}.}  We point the reader to a recent survey of such algorithms, techniques, and implementations for more details \cite{gilbert2014recent}.  

Herein we are interested in deterministic SFT methods with no probability of failing to recover the dominant Fourier series coefficients of $f$. As with randomized techniques, most methods of this kind (see, e.g., \cite{Iw,Ak,Iw-arxiv,morotti2016explicit,plonka_sparse}) focus on the recovery of functions $f$ that are unstructured and (approximately) $n$-sparse in the Fourier domain.  As one might expect, these techniques are generally slower than their randomized counterparts, and the fastest run in $n^2 \log^{\mathcal{O}(1)} N$ time in the unstructured frequency sparse case. Other deterministic SFT algorithms are based on Prony's method (see, e.g., \cite{heidersparse,plonka2014,potts_tasche_volkmer}) and achieve runtimes of $\mathcal{O}(n^3)$ for $n$-sparse input functions. However, Prony-based methods suffer from numerical instabilities for large bandwidths $N$ and noisy input data and thus are not suitable for all applications.   

Note the quadratic runtime of the not Prony-based deterministic methods in $n$.  It is worth mentioning that reducing the quadratic runtime dependence on $n$ for unstructured frequency sparse signals necessitates a similar reduction in the sampling complexity of these deterministic methods which (even when considered independently of the sub-linear runtimes we demand herein) is known to be notoriously difficult (see, e.g., \cite{bourgain2011explicit,foucart2013mathematical,cheraghchi2016nearly}).  This makes meaningful runtime reductions of these methods for periodic functions with unstructured sparsity quite unlikely to occur anytime soon.  However, runtime reductions for functions with structured frequency sparsity should be more tractable.  In this paper we demonstrate this fact by constructing deterministic algorithms which achieve sub-linear runtimes that scale sub-quadratically in sparsity for a wide class of functions with structured frequency support.

\subsection{A General Class of Functions with Structured Frequency Support}

Existing sparse Fourier transform techniques have been applied to many signal processing problems including, e.g., GPS signal acquisition \cite{HAKI}, analog-to-digital conversion \cite{laska2006random,yenduri2012low}, and wideband communication/spectrum sensing \cite{YG2012,HSAHK}.  In all of these applications the signals under consideration are generally manmade and, therefore, structured in Fourier space.  
Herein we will in particular focus on periodic functions $f$ whose dominant Fourier series coefficients are all associated with integer frequencies belonging to sets $S_1, \dots, S_n \subset \left( - \lceil N/2 \rceil, \lfloor N/2 \rfloor \right] \cap \mathbb{Z}$, each of which is generated by an unknown degree $\leq d$ polynomial $P_j \in \mathbb{Z}[x]$.  More specifically, we will assume that each set $S_j$ is given by
\begin{equation}
S_j := \left \{ P_j(k) ~:~ k=1,\dots,B'_j \in \mathbb{N} \right \},
\label{equ:Poly_Structured_Support}
\end{equation}
where $0 < B'_j \leq B \ll N$ always holds for some support set cardinality upper bound $B \in \mathbb{N}$.  Perhaps the simplest class of structured frequency sparse functions of this type are the \textit{block frequency sparse functions} for which each $P_j(x) = x + a_j$ for some $a_j \in \left( - \lceil N/2 \rceil, \lfloor N/2 \rfloor - B \right] \cap \mathbb{Z}$.

Though our main results will concern the relatively general setting where our frequency support sets are given by \eqref{equ:Poly_Structured_Support}, in what follows we will pay particular attention to the simpler class of block frequency sparse functions.  Related block Fourier sparse structures appear in many signal processing contexts including, e.g., the reconstruction of multiband signals via blind sub-Nyquist sampling \cite{fengbresler1996,eldar2008_eff_sampl_sparsewideband,eldar2009_blind_multiband_signal_rec,eldar2010_theorytopractice,eldar2011_xampling_analog_digital_subnyquist}.  This class of block Fourier sparse functions also appears in related numerical methods for the rapid approximation of functions which exhibit sparsity with respect to other orthonormal basis functions.  For example, one can rapidly approximate functions which are a sparse combination of high-degree Legendre polynomials by computing the DFT of samples from a related periodic function which is always guaranteed to be approximately block frequency sparse \cite{hu2015rapidly}.

The importance of block frequency sparse functions has already led several authors to consider deterministic sub-linear time Fourier transforms for this case.  Examples include several approaches which focus on the recovery of periodic functions whose frequency support is confined to just one block \cite{plonka_smallsupp,bittens2016,plonka_nonneg} or several blocks \cite{CKSZ17}.  Herein we significantly generalize these first block frequency sparse recovery results by developing new deterministic SFT methods which enjoy recovery guarantees for all structured frequency sparse periodic functions $f$ satisfying both \eqref{equ:Example_Problem} and~\eqref{equ:Poly_Structured_Support}.  In particular, the methods proposed herein can rapidly recover block frequency sparse functions whose frequency support contains any given number of blocks.

\subsection{Notation and Setup}
\label{sec:Intro_Notation}

We will always consider continuous $2 \pi$-periodic functions $f\colon[0,2\pi]\rightarrow\CC$ with $f(x) = \sum^n_{j=1} \sum_{\omega \in S_j} c_\omega(f) e^{i \omega x}$ where the unknown support sets $S_1, \dots, S_n$ all satisfy \eqref{equ:Poly_Structured_Support}.  We will denote the Fourier series coefficients of any such $f$ by $\cc(f)=\left(c_\omega(f)\right)_{\omega\in\ZZ}$ with
\[
 c_\omega(f):=\frac{1}{2\pi}\int\limits_0^{2\pi}f(x)e^{-i\omega x}dx.
\]
We will also consider perturbations of $f$ by arbitrary $2\pi$-periodic functions $\eta\in L^2([0,2\pi])$ whose Fourier series coefficients $\cc(\eta) \in \ell^1$ and also satisfy $\|\cc(\eta)\|_\infty\leq\eps$ for some $\eps > 0$. 

We will further denote by $\mathbf{c}(N)\in\CC^N$ the restriction of the sequence $\cc(f+\eta)$ to the frequencies contained in $\intvl$, and by $\cc(N,\ZZ)$ the embedding of $\cc(N)$ into $\CC^\ZZ$:
\[
 \left(\cc(N,\ZZ)\right)_\omega=\begin{cases}
                               c_\omega(f+\eta), &\quad \omega\in\intvl, \\
                               0, &\quad\text{otherwise.}
                              \end{cases}
\]
A Fourier coefficient $c_\omega:= c_\omega(f+\eta) \in\CC$ will be called \emph{significantly large} if $|c_\omega|>\eps$. Similarly, a frequency $\omega\in\ZZ$ is \emph{energetic} if its corresponding Fourier coefficient $c_\omega$ is significantly large.

For any vector $\mathbf{x}\in\CC^{|I|}$ with index set $I$, and subset $R \subseteq I$, we define the vector $\mathbf{x}_R\in\CC^{|I|}$ by
 \[
  (\mathbf{x}_R)_i=\begin{cases}
                    x_i, \quad\text{if } i\in R,\\
                    0, \quad\text{otherwise}
                   \end{cases}
 \]
 for all $i\in I$. Furthermore, we denote by $\mathbf{0}_M\in\CC^M$ the vector consisting of $M$ zeroes and by $\mathbf{1}_M\in\CC^M$ the  vector consisting of $M$ ones.
 
Finally, for any $s < |I|$ we will let the subset $R_s^{\text{opt}}\subset I$ be the, in lexicographical order, first $s$-element subset such that  $|x_j|\geq|x_k|$ for all $j \in R_s^{\text{opt}}$ and $k\in I\backslash R_s^{\text{opt}}$. Thus, $R_s^{\text{opt}}$ contains the indices of $s$ entries of $\x$ with the largest magnitudes. While choosing $s$ entries with the largest magnitude might not be unique, $R_s^{\text{opt}}$ is unique. To simplify notation we set $\x_s^{\text{opt}}:=\x_{R_s^{\text{opt}}}$.  We will also say, e.g., that $\left( \mathbf{c}(N) \right)_s^{\text{opt}} = \cc\opt_{s}(N)$.

Throughout the remainder of this paper we will always consider samples to be taken from $f+\eta$ so that we are recovering a function which is potentially both non-sparse in Fourier domain and unstructured in its dominant frequency support. However, if, for example, the nonzero Fourier coefficients of $f$ all satisfy $|c_\omega(f)|>2\eps$, then the structured frequency sparsity of $f$ guarantees that 
$$ \left\{ \omega ~:~ |c_\omega(f+\eta)| > \epsilon \right\} \cap R_{Bn}\opt(f+\eta) \subseteq S:= \bigcup^n_{j=1} S_j.$$ 
 It is exactly this type of consideration which will allow us to obtain near-optimal best $Bn$-term approximation guarantees for $f + \eta$ via our deterministic SFT methods below.

\subsection{Results}\label{ch:intro_results}

As previously mentioned, we will confine our reconstruction results to the class of periodic functions, $f$, with structured frequency support satisfying both \eqref{equ:Example_Problem} and~\eqref{equ:Poly_Structured_Support} above.  Momentarily ignoring the structure of the support sets, $S_j$, given in \eqref{equ:Poly_Structured_Support} one can see that each such $f$ is approximately $Bn$-sparse.  As a result, it can be recovered in $B^2 n^2 \log^{\mathcal{O}(1)} N$-time using the best deterministic SFTs for unstructured sparsity \cite{Iw,Iw-arxiv}.  Herein we obtain the following improved deterministic recovery result by taking the structure of the support sets \eqref{equ:Poly_Structured_Support} into account.  It is a simplified corollary of Theorem~\ref{thm:main} in \S\ref{sec:GeneralAlgResults}.

\begin{thm}\label{thm:main_intro}
Let $f, \eta \in L^2([0,2\pi])$ be as in \S\ref{sec:Intro_Notation}, where $n$ is the number of polynomials of degree at most $d$ which are evaluated at most $B$ times to obtain the energetic frequencies in $S_1, \dots, S_n$.  In addition, assume for simplicity that $\cc_\omega(f+\eta) = 0$ for all $\omega \notin \intvl$.  In this case Algorithm~\ref{alg:fourierapprox} below is guaranteed to always return a sparse $N$-length vector $\x_R$ of Fourier coefficient estimates that satisfies 
\begin{equation}
 \|\cc(N)- {\x_R}\|_2\leq\left\|\cc(N)-\cc\opt_{Bn}(N)\right\|_2+\sqrt{B}\left(\eps+\frac{3}{d\sqrt{n}}\left\|\cc(N)-\cc\opt_{2Bn}(N)\right\|_1\right)
\label{equ:Main_thm_Error_Bound}
\end{equation}
when given access to 
\[
 \mathcal{O}\left(\frac{Bd^2n^3\log^5N}{\log B\log^2(dn)}\right)
\]
samples from $f+\eta$ on $[0, 2\pi]$.  Furthermore, the runtime of Algorithm~\ref{alg:fourierapprox} is always
\[
 \mathcal{O}\left(\frac{Bd^2n^3\log^5N}{\log^2(dn)}\right).
\]
\end{thm}

Note that algorithm mentioned in Theorem~\ref{thm:main_intro} will outperform existing deterministic SFTs with respect to runtime on functions with structured frequency support whenever $B \gg d^2 n \log N$.\footnote{Note that, as we are disregarding constant factors and all lower order terms in each of the variables $B$, $d$, $n$ and $N$, the big-$\mathcal{O}$ notation is true for $B,d,n,N\rightarrow\infty$.} Most importantly, it does so while still maintaining a (slightly weakened) $\ell^2/\ell^1$ error guarantee \eqref{equ:Main_thm_Error_Bound} of the same type as the error guarantees of many compressive sensing methods \cite{foucart2013mathematical}.  

Of course nothing comes for free.  The error guarantee \eqref{equ:Main_thm_Error_Bound} is only really meaningful in the setting where the function with structured frequency support, $f$, dominates the arbitrary noise $\eta$ in $f+\eta$.  If, for example, $|c_\omega(f)|<2\eps$ for all $\omega\in S = \bigcup^n_{j=1} S_j$, then $f+\eta$ might not be approximated well by a function with structured frequency sparsity anymore.  In such cases the runtime of Algorithm~\ref{alg:fourierapprox} in Theorem~\ref{thm:main_intro} will still be fast, but at the expense of the right hand side of \eqref{equ:Main_thm_Error_Bound} being relatively large (due to the $\eps$-term).  As a consequence, one can see that Theorem~\ref{thm:main_intro} only provides a meaningful computational improvement over standard deterministic SFTs in the case where, e.g., both $B \gg d^2 n \log N$ and 
$$\eps := \| \cc(\eta) \|_{\infty} \lesssim \frac{1}{d\sqrt{n}} \left\|\cc(N)-\cc\opt_{Bn}(N)\right\|_1 \leq \frac{1}{d\sqrt{n}} \sum_{\omega \notin S} |c_\omega(\eta)|$$ are true.

If one focuses on the more restrictive case of block frequency sparse functions, where all support sets $S_j$ in \eqref{equ:Poly_Structured_Support} are generated by evaluating $n$ linear, monic polynomials at $B$ consecutive points, the methods developed herein also provide the following simplified result.  It is a corollary of Theorem~\ref{thm:side} in \S\ref{sec:Alg_for_Special_Props}.

\begin{thm}\label{thm:main_intro_blockSparse}
Let $f, \eta \in L^2([0,2\pi])$ be as in \S\ref{sec:Intro_Notation} and let further $f$ be block frequency sparse.  In addition, assume for simplicity that $\cc_\omega(f+\eta) = 0$ for all $\omega \notin \intvl$.  In this case the variant of Algorithm~\ref{alg:fourierapprox} presented in \S\ref{sec:Alg_for_Special_Props} is guaranteed to always return a sparse $N$-length vector ${\x_R}$ of Fourier coefficient estimates that satisfies 
\[
 \|\cc(N)- {\x_R}\|_1\leq 4\left\|\cc(N)-\cc\opt_{Bn}(N)\right\|_1+2Bn\eps
\]
when given access to 
\[
 \mathcal{O}\left(\frac{Bn^2\log^4N}{\log^2n}\right)
\]
samples from $f+\eta$ on $[0, 2\pi]$.  Furthermore, the runtime of the algorithm is always
\[
 \mathcal{O}\left(\frac{Bn^2\log B\log^4N}{\log^2n}\right).
\]
\end{thm}

Inspecting Theorem~\ref{thm:main_intro_blockSparse} above one can see that it always provides a theoretical runtime improvement over existing $B^2 n^2 \log^{\mathcal{O}(1)} N$-time methods for unstructured sparsity when applied to block frequency sparse functions.  Moreover, a (slightly weakened) $\ell^1/\ell^1$ sparse approximation error guarantee is obtained.  As above, we note that this result represents a significant improvement over existing techniques for this class of periodic functions as long as $\eps$ is sufficiently small (i.e., as long as $f + \eta$ is sufficiently well approximated by a block frequency sparse function).

\subsection{Techniques and Overview}

The deterministic SFT algorithms introduced in \cite{Iw} implicitly construct compressive sensing matrices $\mathcal{M} \in \{ 0,1\}^{m \times N}$ with $m \ll N$ which have several useful properties, including $(i)$ the restricted isometry property, $(ii)$ they are the adjacency matrices of highly unbalanced expander graphs, and $(iii)$ they are $d$-disjunct group testing matrices.\footnote{See \cite{bailey2012design} for additional details about these matrices and all of their remarkable properties.}  In addition to these properties, the matrices $\mathcal{M}$ also interact well with the Fourier basis in the following sense.  Let $\F \in \mathbb{C}^{N \times N}$ be a discrete Fourier transform matrix. Then, the matrix product $\mathcal{M} \F$ is guaranteed to be highly sparse, with fewer than $m$ columns containing nonzero entries.   

It is precisely this collection of properties of $\mathcal{M}$ which ultimately allows for the development of the improved deterministic SFT algorithms presented in \cite{Iw-arxiv}.  To get some intuition for how this works, one can consider the recovery of the approximately sparse Fourier coefficients $\cc(N) \in \CC^N$ using only the measurements $\mathcal{M} \left( \cc(N) \right) \in \CC^m$.  The properties of $\mathcal{M}$ make it clear that such recovery is possible via, e.g., standard compressive sensing methods \cite{foucart2013mathematical}.  In fact, with more work one can show that the special properties of $\mathcal{M}$ allow for the recovery of $\cc(N)$ in just $m \log^{\mathcal{O}(1)}m$-time, and without compromising the error guarantees one generally expects from compressive sensing algorithms such as Basis Pursuit.  In addition, only a small number of samples from $f + \eta$ are required, since $\mathcal{M}\F$ is highly sparse, and 
$$\mathcal{M} \cc(N) = \left(\mathcal{M} \F \right) \left( \F^{-1} \cc(N) \right) = \left(\mathcal{M} \F \right) {\x}$$
where ${\x}$ has $x_j= (f + \eta) \left( \frac{2 \pi j}{N} \right)$.  Thus, just a few entries of ${\x}$ have to be observed in order to obtain the necessary measurements $\mathcal{M} \left( \cc(N) \right)$.

In this paper we build on \cite{Iw,Iw-arxiv} by augmenting the number theoretic constructions of the matrices $\mathcal{M}$ above in a way which allows us to benefit from structured frequency support while simultaneously preserving all the of properties of $\mathcal{M}$ needed in order to maintain extremely fast (i.e., sub-linear) runtimes.  Intuitively, this is accomplished by augmenting a well chosen measurement matrix $\mathcal{M}$ from \cite{Iw-arxiv} with a set of several additional vectors $( {\mathbf u}_j )^{M}_{j=1} \subset \{ 0, 1\}^N$ as follows.  Let $\circ$ denote the Hadamard product and $\cast$ the row-wise Hadamard product, where the first $\kappa$ rows of $A\cast B$ are given as the Hadamard product of all rows of $A$ with the first row of $B$, the second $\kappa$ rows as the Hadamard product of all rows of $A$ with the second row of $B$ and so forth. Below we will utilize a set of new measurement matrices $\mathcal{M} \cast {\mathbf u}_1, \dots, \mathcal{M} \cast {\mathbf u}_M \in \{ 0,1 \}^{m \times N}$.  Collectively, these new matrices are then shown to still allow all of the measurements $\left( \mathcal{M} \cast {\mathbf u}_1 \right) \left( \cc(N) \right), \dots, \left( \mathcal{M} \cast {\mathbf u}_M \right) \left( \cc(N) \right) \in \CC^m$ to be computed using just a few samples from $f + \eta$.  Furthermore, when $f + \eta$ is structured frequency sparse, it is shown that ${\mathbf u}_j \circ \cc(N)$ will be guaranteed to be significantly more sparse than $\cc(N)$ for most of the values of $j = 1, \dots, M$.  Hence, the deterministic SFT methods from \cite{Iw,Iw-arxiv} will allow 
each such ${\mathbf u}_j \circ \cc(N)$ to be recovered using the measurements 
$$\left( \mathcal{M} \cast {\mathbf u}_j \right) \left( \cc(N) \right) = \mathcal{M} \left( {\mathbf u}_j \circ \cc(N) \right)$$
much faster than one can deterministically recover $\cc(N)$ all at once using the same techniques.  

Finally, the structure of the vectors $( {\mathbf u}_j )^{M}_{j=1} \subset \{ 0, 1\}^N$ is then used to rapidly and accurately reconstruct $\cc(N)$ from the set of its partial reconstructions of $\left ( {\mathbf u}_j \circ \cc(N) \right )^M_{j=1}$.  Here it becomes crucial to deal with the fact that the partial reconstructions of ${\mathbf u}_j \circ \cc(N)$ are incorrect for some values of $j$.  Thankfully, median arguments adapted from earlier SFT algorithms \cite{GGIMS,GMS} allow this to be handled easily by simply using enough vectors $( {\mathbf u}_j )^{M}_{j=1}$ in order to guarantee that the majority of the values of $j$ provide good results.  It then just remains to modify the reconstruction procedure from \cite{Iw-arxiv} in order to rapidly recover $\cc(N)$ from the partial reconstructions of $\left ( {\mathbf u}_j \circ \cc(N) \right )^M_{j=1}$.

The remainder of the paper is organized as follows:  In \S\ref{sec:Preliminaries} the vectors $( {\mathbf u}_j )^{M}_{j=1}$ discussed above are constructed, and it is proven that ${\mathbf u}_j \circ \cc(N)$ will be approximately sparse for the majority of the ${\mathbf u}_j$ whenever $f + \eta$ exhibits sufficiently structured frequency support.  Next, a deterministic reconstruction algorithm is developed for functions with structured frequency support in \S\ref{sec:GeneralAlgResults},  and Theorem~\ref{thm:main_intro} is proven.  These results are then improved for the simpler class of block frequency sparse signals in \S\ref{sec:Alg_for_Special_Props}, and Theorem~\ref{thm:main_intro_blockSparse} is proven. Finally, the methods developed for block frequency sparse functions are empirically evaluated in \S\ref{sec:Numerical_Eval}.  The paper then concludes with a short discussion of future work in \S\ref{sec:Conclusion}.

\section{Preliminaries}  
\label{sec:Preliminaries}    
\begin{diss}
We will always consider a continuous periodic function $f\colon[0,2\pi]\rightarrow\CC$ which is bandlimited, i.e. for its finite Fourier transform 
$\cc(f)=\left(c_\omega\right)_{\omega\in\ZZ}$ with
\[
 c_\omega:=c_\omega(f):=\frac{1}{2\pi}\int\limits_0^{2\pi}f(x)e^{-i\omega x}dx
\]
we have that $c_\omega=0$ for all $\omega\notin\intvl$ for some large $N\in\NN$. Note that $N$ does not have to be a power of 2. Further, we consider a $2\pi$-periodic noise function $\eta$ that satisfies $\cc(\eta)\in\ell^1$ and $\|\cc(\eta)\|_\infty\leq\eps$ for some $\eps>0$ and may have nonzero Fourier coefficients outside of $\intvl$. 

We denote by $\mathbf{c}(N)\in\CC^N$ the restriction of the vector $\cc(f+\eta)$ to the frequencies contained in $\intvl$ and by $\cc(N,\ZZ)$ the embedding of $\cc(N)$ into $\CC^\ZZ$:
\[
 \left(\cc(N,\ZZ)\right)_\omega=\begin{cases}
                               c_\omega, &\quad \omega\in\intvl, \\
                               0, &\quad\text{otherwise.}
                              \end{cases}
\]
With the help of the Fourier transform $\widehat{f}$ of $f$, evaluated at a frequency $\omega$, we obtain:
\begin{align*}
 \widehat{f}(\omega)=&\frac{1}{2\pi}\int_0^{2\pi}\sum_{\nu\in\ZZ}c_\nu e^{i\nu x}\cdot e^{-i\omega x}dx \\
 =&\sum_{\nu\neq\omega}\frac{1}{2\pi}\int_0^{2\pi}c_\nu e^{i\nu x}\cdot e^{-i\omega x}dx+\frac{1}{2\pi}\int_0^{2\pi}c_\omega=c_\omega.
\end{align*}
\begin{dfn}
 Let $\eps>0$ be a given accuracy level. Let $f\colon[0,2\pi]\rightarrow\CC$ be defined as
 \[
  f(x)=\sum_{\omega\in\ZZ}c_we^{i\omega x}.
 \]
 Then a Fourier coefficient $c_\omega\in\CC$ is \emph{significantly large} if $|c_\omega|>\eps$. A frequency $\omega\in\ZZ$ is \emph{energetic} if its corresponding Fourier coefficient $c_\omega$ is significantly large.
\end{dfn}
\begin{dfn}[Notation]\label{dfn:notation1}
 Here, $\NN$ will always denote the positive natural numbers, i.e. $\NN=\{1,2,3,\dotsc\}$. Let $M\in\NN$, $I=\{0,\dotsc,M-1\}$ and $R\subset I$. For a vector $\mathbf{x}\in\CC^{|I|}$ we define the vector $\mathbf{x}_R\in\CC^{|I|}$ by
 \[
  (\mathbf{x}_R)_i=\begin{cases}
                    x_i, \quad\text{if } i\in R,\\
                    0, \quad\text{otherwise}
                   \end{cases}
 \]
 for all $i\in I$. Furthermore we denote by $\mathbf{0}_M\in\CC^M$ the vector consisting of $M$ zeroes and by $\mathbf{1}_M\in\CC^M$ the  vector consisting of $M$ ones.
 
 Define for a $D<|I|$ the subset $R_D^{\text{opt}}\subset I$ to be the, in lexicographical order, first $D$ element subset such that  $|x_j|\geq|x_k|$ for all $j\in R_D^{\text{opt}}$ and $k\in I\backslash R_D^{\text{opt}}$. Then $R_D^{\text{opt}}$ contains the indices of  the $D$ entries of $\x$ that have the largest magnitudes. While choosing $D$ entries with the largest magnitude might not be unique,  $R_D^{\text{opt}}$ is unique. For reasons of clarity we set $\x_D^{\text{opt}}:=\x_{R_D^{\text{opt}}}$.
\end{dfn}
% From now on all functions will always have a sparse Fourier transform, meaning that only a few Fourier coefficients are significantly large, where the energetic frequencies comply with a certain structure:
\end{diss}
First, we formally define the notion of polynomially structured sparsity that was already mentioned in (\ref{equ:Poly_Structured_Support}) in \S\ref{ch:intro}.
\begin{dfn}[$P(n,d,B)$-structured Sparsity]
 Let $B,d,n,N\in\NN$ such that $d<B<N$ and let $P_1,\dotsc,P_n\in\ZZ[x]$ be non-constant polynomials of degree at most $d$ with
 \[
  P_j(x)=\sum_{k=0}^da_{jk}x^k,
 \]
 where $a_{jk}\in\intvl$ such that for all $j\in\{1,\dotsc,n\}$ and $x\in\{1,\dotsc,B\}$ we have $P_j(x)\in\intvl$. Define the $n$ \emph{support sets}
 \[
  S_j:=\{P_j(x):x\in\{1,\dotsc,B\}\}
 \]
 and let $S:=\bigcup_{j=1}^nS_j$. A $2\pi$-periodic function $f\colon[0,2\pi]\rightarrow\CC$ is \emph{P(n,d,B)-structured sparse} if it is of the form
 \[
  f(x)=\sum_{\omega\in S}c_\omega(f) e^{i\omega x}
 \]
 for some vector of Fourier coefficients $(c_\omega(f))_{\omega\in S}\in\CC^{Bn}$.
\end{dfn}
This means that the at most $Bn$ energetic frequencies of the function $f$ are generated by evaluating $n$ polynomials of degree at most $d$ with integer coefficients at $B$ points.

Our aim in this paper is to develop a sublinear-time Fourier algorithm for $P(n,d,B)$-structured sparse input functions, based on ideas introduced in \cite{Iw} and \cite{Iw-arxiv}. One important concept for our method is that of a \emph{good hashing prime}; a prime modulo which not all frequencies in a support set $S_j$ are hashed to the same residue.
\begin{dfn}\label{dfn:wellhashed}
 Let $f$ be a $P(n,d,B)$-structured sparse function with support set $S=\bigcup_{j=1}^nS_j$ generated by some polynomials $P_1,\dotsc,P_n$. Then a prime $p>B$ \emph{hashes} a support set $S_j$ \emph{well} if
 \[
  |\{\omega\mods p:\omega\in S_j\}|>1.
 \]
\end{dfn}
\begin{lem}\label{lem:wellhashed_iff}
  Let $f$ be a $P(n,d,B)$-structured sparse function with support set $S=\bigcup_{j=1}^nS_j$ defined by some polynomials $P_1,\dotsc,P_n$. Then a prime $p>B$ hashes a support set $S_j$ with generating polynomial 
 \[
  P_j(x)=\sum_{k=0}^da_{jk}x^k
 \]
 well if and only if there exists a non-constant coefficient $a_{jk}$, $k\neq0$, with $p\nmid a_{jk}$.
\end{lem}
\begin{proof}
 Assume $p|a_{jk}$ for all $k\in\{1,\dotsc,d\}$. Then we have for all $x\in\{1,\dotsc,B\}$ that
 \[
  P_j(x)=\sum_{k=0}^da_{jk}x^k\equiv a_{j0}\mods p \quad \Rightarrow \quad |\{\omega\mods p:\omega\in S_j\}|=1,
 \]
 so $p$ does not hash $S_j$ well. If, on the other hand, $p$ does not hash $S_j$ well, then 
 \[
  |\{\omega\mods p:\omega\in S_j\}|=1 \quad\Rightarrow\quad P_j(y)\equiv P_j(z)\mods p \quad \forall y,z\in\{1,\dotsc,B\}.
 \]
 This means that for fixed $y\in\{1,\dotsc,B\}$ the polynomial 
 \[
  Q(x):=P_j(x)-P_j(y)=\sum_{k=0}^da_{jk}x^k-P_j(y)
 \]
 of degree $d$ has $B>d$ zeroes modulo $p$. Thus $Q$ is the zero polynomial modulo $p$, and 
 \[
  p|\left(a_{j0}-P_j(y)\right) \quad\text{and}\quad p|a_{jk} \quad\forall j\in\{1,\dotsc,d\}. 
 \]
\end{proof}
For a good hashing prime and a $P(n,d,B)$-structured sparse function we can bound the number of frequencies that are hashed to the same residue.
\begin{lem}\label{lem:wellhashed}
 Let $f$ be a $P(n,d,B)$-structured sparse function with support set $S=\bigcup_{j=1}^nS_j$ defined by some polynomials $P_1,\dotsc,P_n$. If  a support set $S_j$ is hashed well by a prime $p>B$, then
 \begin{enumerate}[label=(\roman*)]
  \item $P_j$ is not constant modulo $p$ and
  \item $|\{\omega\equiv \nu\mods p:\omega\in S_j\}|\leq d$ for all residues $\nu\in\{0,\dotsc,p-1\}$.
 \end{enumerate}
\end{lem}
\begin{proof}
 It is clear that $P_j$ is not constant modulo $p$ if $|\{\omega\mods p:\omega\in S_j\}|>1$. Assume now that  $|\{\omega\equiv \nu\mods p:\omega\in S_j\}|> d$ for some $\nu\in\{0,\dotsc,p-1\}$. Since all elements of $S_j$ are generated by 
\begin{diss}
 the polynomial 
 \[
  P_j(x)=\sum_{k=0}^da_{jk}x^k,
 \]
\end{diss}
\begin{paper}
 evaluating $P_j$ at $B$ points,
\end{paper}
 we find for a $y\in\{1,\dotsc,B\}$ with $P_j(y)\equiv \nu\mods p$ that
 \[
  P_j(y)\equiv P_j(z) \mods p
 \]
 for $d$ distinct choices of $z\in\{1,\dotsc,B\}\backslash\{y\}$. Then the polynomial
\begin{diss}
 \[
  Q(x):=P_j(x)-P_j(y)=\sum_{k=1}^{d}a_{jk}x^k+a_{j0}-P_j(y)
 \]
\end{diss}
\begin{paper}
 $Q(x):=P_j(x)-P_j(y)$
\end{paper}
 has at least $d+1$ zeroes modulo $p$,
\begin{diss}
 . As $Q$ is a polynomial of degree at most $d$ modulo $p$, this is a contradiction, so (ii) holds.
\end{diss}
\begin{paper}
which is a contradiction, so (ii) holds.
\end{paper}
\end{proof}
Let us now assume that there exists a prime $p>B$ that hashes all support sets $S_1,\dotsc,S_n$ of a $P(n,d,B)$-structured sparse function well. Then the restriction of any $S_j$ to the frequencies congruent to $\nu$ modulo $p$ is at most $d$-sparse for all residues $\nu\in\{0,\dotsc,p-1\}$. Consequently, the restriction of $S$ to these frequencies is at most $dn$-sparse. 

 In our setting of $P(n,d,B)$-structured sparse input functions we want to apply the SFT algorithm in \cite{Iw-arxiv} (Algorithm 3) to the restrictions of the function to frequencies congruent to $\nu$ modulo $u$ for all residues $\nu$, where $u$ is a prime that hashes all support sets well, since these restrictions are at most $dn$-sparse.

In general, finding a single well-hashing prime $u$ for all support sets is not possible without further information on the generating polynomials. However, we can use the observations presented in the remainder of \S\ref{sec:Preliminaries} to find $M$ primes such that the majority of them hashes all support sets well. As these methods, as well as Algorithm 3 in \cite{Iw-arxiv}, rely heavily on the Chinese Remainder Theorem, we state it here as a reminder (see \cite{lang_algebra}).
\begin{thm}[Chinese Remainder Theorem (CRT)]\label{thm:crt}
 Let $n_1,\dotsc,n_m$ be pairwise relatively prime integers and $N\leq\prod_{j=1}^{m}n_j$. Then the system of simultaneous congruencies $x\equiv a_1 \mods n_1,\dotsc,x\equiv a_m\mods n_m$ has a unique solution modulo $N$. 
\end{thm}
\begin{dfn}[Enumeration of the Natural Primes]
 For $j\in\NN$ denote by $p_j$ the $j$-th natural prime number. Additionally, let $p_0=1$. Then,
 \[
  p_0=1,\quad p_1=2,\quad p_2=3,\quad p_3=5,\quad p_4=7,\quad \dotsc\quad .
 \]
\end{dfn}
In order to find primes for which the restriction of the frequencies to any residue is sparse, we first need to define the notion of separation.
\begin{dfn}[Separation]
 Let $k\in\NN$ and $\omega_1,\dotsc,\omega_k\in\ZZ$. An integer $n\in\NN$ is said to \emph{separate}  
 $\omega_1,\dotsc,\omega_k$ if
 \[
  \omega_j\mods n\neq\omega_l\mods n \quad\forall j,l\in\{1,\dotsc,k\},\, j\neq l.
 \]
\end{dfn}
The following result about separating primes has been shown in \cite{Iw-arxiv}.
\begin{lem}\label{lem:u_m}
 Let $E\in\NN$ and $u_1:=p_r$ for some $r\in\NN$. Set $M=2\cdot E\cdot\lfloor\log_{u_1}N\rfloor+1$. Choose $M-1$ further primes with $u_1<\dotsb<u_M$ and let $T\subset\intvl$ with $|T|\leq E$. Then more than $\frac{M}{2}$ of the $u_m$ %primes  $u_1,\dotsc,u_M$ 
 separate every $x\in\intvl$ from all $t\in T\backslash\{x\}$.
\end{lem}
In the next lemma we prove that, for a suitable $M$, it suffices to find $M$ primes such that more than half of them separate the leading coefficients of the frequency generating polynomials from 0 at the same time, in order to guarantee that more than half of them hash all support sets well.
\begin{lem}\label{lem:umwellhashed}
 Let $f$ be $P(n,d,B)$-structured sparse with support set $S=\bigcup_{j=1}^nS_j$ defined by some polynomials $P_1,\dotsc,P_n$, and set $E=n+1$. Let $M$ primes $B<u_1<\dotsb<u_M$ be given as in Lemma \ref{lem:u_m}. Then more than $\frac{M}{2}$ of the $u_m$ hash all $n$ support sets $S_1,\dotsc,S_n$ well.
\end{lem}
\begin{proof}
 Let $T$ be the set consisting of the distinct leading polynomial coefficients,
 \[
  T:=\left\{a_{j,\deg(P_j)}: P_j(x)=\sum_{k=0}^{\deg(P_j)}a_{jk}x^k,\,j\in\{1,\dotsc,n\}\right\}.
 \]
 Then $a_{j,\deg(P_j)}\neq 0$ for all polynomials and, since $|T\cup\{0\}|\leq E$, by Lemma \ref{lem:u_m} more than $\frac{M}{2}$ of the $u_m$ separate every element of $\intvl$ from all other elements of $T\cup\{0\}$, i.e., from all distinct leading polynomial coefficients  and from 0. 
 
 Let $p$ be such a prime. Assume that there exists a support set $S_j$ that is not well hashed by $p$, so that we have
 \[
  \{\omega\mods p:\omega\in S_j\}=\{\nu\}
 \]
 for some residue $\nu\in\{0,\dotsc,p-1\}$. Then the polynomial $P_j$ that generates $S_j$ satisfies 
 \[
  P_j(x)-\nu\equiv0\mods p \quad \forall x\in\{1,\dotsc,B\}.
 \]
 Consider now the polynomial $Q(x):=P_j(x)-\nu$ modulo $p$. It is a polynomial of degree at most $d$ with $B>d$ zeroes, so it has to be the zero polynomial modulo $p$, meaning that
 \[
  p|a_{jk}\quad \forall k\in\{1,\dotsc,d\} \quad\text{and}\quad p|(a_{j0}-\nu).
 \]
 Since $p$ separates $a_{j,\deg(P_j)}$ from 0 and the other leading coefficients, we find that 
 \[
  a_{j,\deg(P_j)}\equiv 0 \mods p \quad\text{and}\quad a_{j,\deg(P_j)}\not\equiv t\mods p \quad 
  \forall t\in (T\cup\{0\})\backslash\{a_{j,\deg(P_j)}\}.
 \]
 This is only possible if $a_{j,\deg(P_j)}=0$, which is a contradiction. Thus we obtain that
 \[
  |\{\omega\mods p:\omega\in S_j\}|>1,
 \]
 so $p$ hashes all $S_j$ well. Consequently, all of the more than $\frac{M}{2}$ primes $u_1,\dotsc,u_M$ that separate the leading coefficients from one another and from $0$ hash all support sets well.
\end{proof}
These lemmas imply that applying Algorithm 3 in \cite{Iw-arxiv} to the restrictions of the input function to frequencies congruent to $\nu$ modulo $M$ primes as defined in Lemma \ref{lem:u_m} yields the correct frequencies and Fourier coefficients in the majority of the cases, since the algorithm works well as long as the input function is sparse enough.

\section{Algorithm for Polynomially Structured Sparse Functions} 
\label{sec:GeneralAlgResults}
Before we can begin to develop our algorithm for polynomially structured sparse functions, we need to develop the notation necessary in order to apply Algorithm 3 in \cite{Iw-arxiv} to the frequency restrictions.
\subsection{Measurement Matrices}
\begin{dfn}[Notation]\label{dfn:notation2}
Let $f\colon[0,2\pi]\rightarrow\CC$ be $P(n,d,B)$-structured sparse with bandwidth $N$. Let $B<u_1<\dotsb<u_M$ be prime and $s_1<\dotsb<s_K$ pairwise relatively prime natural numbers such that there exist $L$ natural numbers $t_1<\dotsb<t_L<s_1$ which satisfy that the set
 \[
  \{t_1,\dotsc,t_L,s_1,\dotsc,s_K,u_1,\dotsc,u_M\}
 \]
 is pairwise relatively prime and that
 \[
  \prod_{l=1}^Lt_l\geq\frac{N}{s_1u_1}.
 \]
 We set $\kappa:=\sum_{k=1}^K s_k$, $\lambda:=1+\sum_{l=1}^Lt_l$ and $\mu:=\sum_{m=1}^M u_m$. Further, we define
 \[
  q:=\lcm(N,s_1,\dotsc,s_K,t_1,\dotsc,t_L,u_1,\dotsc,u_M).
 \]
\end{dfn}
From now on we always assume that occurring natural numbers $q,s_1,\dotsc,s_K,t_1,\dotsc,t_L,$ $u_1,\dotsc,u_M$ comply with Definition \ref{dfn:notation2}. In the algorithm we will develop in this section the numbers $K$, $L$ and $M$ will depend on the number $n$ of evaluated polynomials, their maximal degree $d$, the number $B$ of evaluations and the bandwidth $N$. We show in Remark \ref{rem:estimates} how $K$, $L$ and $M$ can be chosen. 
\begin{dfn}
 For a given $2\pi$-periodic function $f$ and $m\in\NN$ define the sample vector $\A_m\in\CC^m$ by
 \[
  \A_m(j):=f\left(\frac{2\pi j}{m}\right) \quad \forall j\in\{0,\dotsc,m-1\}.
 \]
\end{dfn}
\begin{dfn}[Discrete Fourier Transform]
For $m\in\NN$ we denote the discrete Fourier transform of a vector $\x\in\CC^m$ by
\[
 \widehat{\x}=\F_m\x,
\]
where $\F_m:=\frac{1}{m}\left(\omega_m^{jk}\right)_{j,k=0}^{m-1}\in\CC^{m\times m}$ is the $m$-th Fourier matrix and $\omega_m:=e^{\frac{-2\pi i}{m}}$ is the $m$-th primitive root of unity.
\end{dfn}
In order to apply Algorithm 3 in \cite{Iw-arxiv} to the restrictions to frequencies that are congruent to $\nu$ modulo $u_m$ for all residues $\nu\in\{0,\dotsc,u_m-1\}$ and all $m\in\{1,\dotsc,M\}$, we need to transform the vector $\widehat{\A_q}$ into a matrix with sparse columns whose entries correspond to the frequencies that are congruent to $\nu$ modulo $u_m$. For this purpose and also for later use we recall the definition of the row-wise Hadamard tensor product.
\begin{dfn}[Row-wise Hadamard Product]
Let $\A\in\CC^{\kappa\times m}$, $\B\in\CC^{\lambda\times m}$. Then the row-wise Hadamard product $\A\circledast \B\in\CC^{(\kappa\cdot\lambda)\times m}$ is given by
\[
 (\A\circledast \B)_{jl}=\A_{j\mods\kappa,l}\cdot \B_{\frac{j-(j\mods \kappa)}{\kappa},l}, \quad j\in\{0,\dotsc,\kappa\lambda-1\}, \, l\in\{0,\dotsc,m-1\},
\]
i.e., the first $\kappa$ rows are given as the Hadamard product of all rows of $\A$ with the first row of $\B$, the second $\kappa$ rows as the Hadamard product of all rows of $\A$ with the second row of $\B$ and so forth.
% are obtained by component wise multiplication of all rows of $A$ with the first row of $B$, the second $\kappa$ rows by multiplication of all rows of $A$ with the second row of $B$ and so forth.
\end{dfn}
% Example for the dissertation.
\begin{diss}
\begin{ex}
 For $\kappa=2$, $\lambda=3$ and $m=3$ consider the row tensor product of the matrices
 \[
  A=\begin{pmatrix}
     a_{00} & a_{01} & a_{02} \\
     a_{10} & a_{11} & a_{12}
    \end{pmatrix}
\quad\text{and}\quad B=\begin{pmatrix}
                        b_{00} & b_{01} & b_{02} \\
                        b_{10} & b_{11} & b_{12} \\
                        b_{20} & b_{21} & b_{22}
                       \end{pmatrix}.
\]
Then we obtain
\[
 A\circledast B=\begin{pmatrix}
                 a_{00}b_{00} & a_{01}b_{01} & a_{02}b_{02} \\
                 a_{10}b_{00} & a_{11}b_{01} & a_{12}b_{02} \\
                 a_{00}b_{10} & a_{01}b_{11} & a_{02}b_{12} \\
                 a_{10}b_{10} & a_{11}b_{11} & a_{12}b_{12} \\
                 a_{00}b_{20} & a_{01}b_{21} & a_{02}b_{22} \\
                 a_{10}b_{20} & a_{11}b_{21} & a_{12}b_{22} \\                 
                \end{pmatrix}.
\]
\end{ex}
\end{diss}
\begin{lem}\label{lem:tensorproduct}
 Let $\A\in\CC^{\kappa\times m}$, $\B\in\CC^{\lambda\times m}$. Then every row of $\A\circledast \B$ is given as the row tensor product of a row of $\A$ with a row of $\B$.
\end{lem}
% Proof in the dissertation.
\begin{diss}
\begin{proof}
 Let $i\in\{0,\dotsc,\kappa\lambda-1\}$. Then we find for the $i$-th row of $A\circledast B$:
 \begin{align*}
  (A\circledast B)_{i,\cdot}&=\left((A\circledast B)_{ij}\right)_{j=0}^{n-1}
  =\left(A_{i\mods\kappa,j}\cdot B_{\frac{i-(i\mods \kappa)}{\kappa},j}\right)_{j=0}^{m-1} \\
  &=\left(A_{i\mods\kappa,\cdot}\right)\circledast\left(B_{\frac{i-(i\mods\kappa)}{\kappa},\cdot}\right).
 \end{align*}
\end{proof}
\end{diss}
\begin{dfn}[Measurement Matrices I]\label{dfn:matrix_I}
 For $t_1,\dotsc,t_L,s_1,\dotsc,s_K,u_1,\dotsc,u_M$ from Definition \ref{dfn:notation2} we construct a special $\mu\times N$ measurement matrix $\Mu$, analogously to the measurement matrix concept used in \cite{Iw-arxiv}. The matrix consists of rows of ones and zeroes, where an entry of  a row is one if and only if its column index is congruent to a certain residue modulo $u_m$. Let $m\in\{1,\dotsc,M\}$ and $\nu\in\{0,\dotsc,u_m-1\}$ be a fixed residue modulo $u_m$. Then we define the row $\rr_{u_m,\nu}$ by
 \begin{equation}\label{eq:r}
  (\rr_{u_m,\nu})_j:=\delta\left((j-\nu)\mods u_m\right)=\begin{cases}
                                                 1, \quad\text{if } j\equiv \nu \mods u_m, \\
                                                 0, \quad\text{otherwise,}
                                                \end{cases}
 \end{equation}
 and set
 \[
  \mathcal{M}_{u_1,M}:=\begin{pmatrix*}[l]
                       \rr_{u_1,0} \\
                       \vdots \\
                       \rr_{u_1,u_1-1} \\
                       \rr_{u_2,0} \\
                       \vdots \\
                       \rr_{u_M,u_M-1}
                      \end{pmatrix*}.
%                      =\begin{pmatrix}
%                        I_{u_1} & I_{u_1} & \hdots \\
%                        I_{u_2} & I_{u_2} & \hdots \\
%                        \vdots  & \vdots  & \vdots \\
%                        I_{u_M} & I_{u_M} & \hdots 
%                       \end{pmatrix}.
 \]
%  Furthermore, for $j\in\{0,\dotsc,N-1\}$, we define the $K\times N$-matrix $\mathcal{M}_{s_1,K,j}$ as the matrix obtained by choosing only the $K$ rows of $\mathcal{M}_{s_1,K}$ which have a nonzero entry in the $j$-th column:
%  \[
%   \mathcal{M}_{s_1,K,j}:=\begin{pmatrix}
%                         r_{1,j\mods s_1} \\
%                         r_{2,j\mods s_2} \\
%                         \vdots \\
%                         r_{K,j\mods s_K}
%                        \end{pmatrix}
%  \]
%  Finally, we set $\mathcal{M}'_{s_1,K,j}$ to be the $K\times(N-1)$-matrix given by deleting the $j$-th column of \mathcal{M}_{s_1,K,j}$:
%  \[
%   \mathcal{M}'_{s_1,K,j}:=
%   \begin{pmatrix}
%    (r_{1,j\mods s_1})_0 & \hdots & (r_{1,j\mods s_1})_{j-1} & (r_{1,j\mods s_1})_{j+1} & \hdots & (r_{1,j\mods s_1})_{N-1} \\
%    (r_{1,j\mods s_2})_0 & \hdots & (r_{1,j\mods s_2})_{j-1} & (r_{1,j\mods s_2})_{j+1} & \hdots & (r_{1,j\mods s_2})_{N-1} \\
%    \vdots 		&	 &        		    & \vdots		       & \vdots &      			  \\
%    (r_{1,j\mods s_K})_0 & \hdots & (r_{1,j\mods s_K})_{j-1} & (r_{1,j\mods s_K})_{j+1} & \hdots & (r_{1,j\mods s_K})_{N-1} \\
%   \end{pmatrix}
%  \]
We define the extension $\HHH$ of $\Mu$ to a $\mu\times q$ matrix by extending all rows $\rr_{u_m,\nu}$ to columns indexed by $j\in\{0,\dotsc,q-1\}$, as given by (\ref{eq:r}). Then, since $u_1,\dotsc,u_M$ divide $q$, we have that
\begin{equation*}
\HHH=\left(
\begin{array}{ccccccccccc}
\cline{1-6} \cline{11-11}
\multicolumn{1}{|c}{I_{u_1}} & \multicolumn{1}{|c}{I_{u_1}} & \multicolumn{1}{|c}{I_{u_1}} & \multicolumn{1}{|c}{I_{u_1}} & \multicolumn{1}{|c}{I_{u_1}} & \multicolumn{1}{|c|}{I_{u_1}} & \dots & \dots & \dots & \dots & \multicolumn{1}{|c|}{I_{u_1}} \\ \cline{1-6} \cline{10-11}
\multicolumn{2}{|c}{\multirow{2}{*}{$I_{u_2}$}} & \multicolumn{2}{|c}{\multirow{2}{*}{$I_{u_2}$}} & \multicolumn{2}{|c|}{\multirow{2}{*}{$I_{u_2}$}} & \multirow{2}{*}{\dots} & \multirow{2}{*}{\dots} & \multirow{2}{*}{\dots} & \multicolumn{2}{|c|}{\multirow{2}{*}{$I_{u_2}$}}  \\
\multicolumn{2}{|c}{} & \multicolumn{2}{|c}{} & \multicolumn{2}{|c|}{} & & & & \multicolumn{2}{|c|}{} \\ \cline{1-6} \cline{9-11}
\multicolumn{3}{|c}{\multirow{3}{*}{$I_{u_3}$}} & \multicolumn{3}{|c|}{\multirow{3}{*}{$I_{u_3}$}} & \multirow{3}{*}{\dots} & \multirow{3}{*}{\dots} & \multicolumn{3}{|c|}{\multirow{3}{*}{$I_{u_3}$}} \\ 
\multicolumn{3}{|c}{} & \multicolumn{3}{|c|}{} & & & \multicolumn{3}{|c|}{} \\
\multicolumn{3}{|c}{} & \multicolumn{3}{|c|}{} & & & \multicolumn{3}{|c|}{} \\ \cline{1-6} \cline{9-11}
\multicolumn{3}{c}{\vdots} & & &&&& \multicolumn{3}{c}{\vdots} \\ \cline{1-4} \cline{8-11}
\multicolumn{4}{|c|}{\multirow{4}{*}{$I_{u_M}$}} & \multirow{4}{*}{\dots} & \multirow{4}{*}{\dots} & \multirow{4}{*}{\dots} & \multicolumn{4}{|c|}{\multirow{4}{*}{$I_{u_M}$}} \\ 
\multicolumn{4}{|c|}{} & & & & \multicolumn{4}{|c|}{} \\
\multicolumn{4}{|c|}{} & & & & \multicolumn{4}{|c|}{} \\
\multicolumn{4}{|c|}{} & & & & \multicolumn{4}{|c|}{} \\ \cline{1-4} \cline{8-11}
% \hline
% \multirow{2}{*}{$\mathbf{1}_{10}$}
%  & 
% \mathbf{M}_{11}^{(1)} & 
% \mathbf{M}_{11}^{(1)\ast}\bar{\mathbf{I}}_{5} \\
% &
% \bar{\mathbf{I}}_{5}\mathbf{M}_{11}^{(1)H} &
% \bar{\mathbf{I}}_{5}\mathbf{M}_{11}^{(1)}\bar{\mathbf{I}}_{5}\\
\end{array} 
\right),
\end{equation*}
where $I_{u_m}$ denotes the $u_m\times u_m$ identity matrix.

\begin{paper}
 Further, as in \cite{Iw-arxiv}, we define the $\kappa\times N$ matrix $\Ms$ and the $(\lambda-1)\times N$ matrix $\Mt$, consisting of the rows corresponding to the possible residues modulo all the $s_k$ and $t_l$, respectively,
\[
 \Ms:=\begin{pmatrix}
       \rr_{s_1,0} \\
       \vdots \\
       \rr_{s_K,s_K-1}
      \end{pmatrix} 
 \quad\text{and}\quad
  \Mt:=\begin{pmatrix}
       \rr_{t_1,0} \\
       \vdots \\
       \rr_{t_L,t_L-1}
      \end{pmatrix}.
\]
We set $\Nt$ to be the $\lambda\times N$ matrix whose first row contains only ones and whose other rows are given by $\Mt$,
\[
 \Nt:=\begin{pmatrix}
       \boldsymbol{1}_N \\
       \Mt
      \end{pmatrix},
\]
and define the $(\kappa\cdot\lambda)\times N$ row-wise Hadamard product $\RRR:=\Ms\cast\Nt$ of $\Ms$ and $\Nt$ and its extension $\GGG$ to a $(\kappa\cdot\lambda)\times q$ matrix. Because all $s_k$, $t_l$ and $u_m$ thus have to divide $q$, we set $q=\lcm(N,s_1,\dotsc,s_K,t_1,\dotsc,t_L)$ as the length of the sample vector.
\end{paper}
\end{dfn}
% Example in the dissertation.
\begin{diss}
\begin{ex}
We now give an example for this type of measurement matrix for $u_1=2$, $u_2=3$, $u_3=5$ and $N=30$. For $\mathcal{M}_{2,3}$ we obtain:
\[
 \bordermatrix{\boldsymbol{j\in\{0,\dotsc,29\}} & \boldsymbol{0} & \boldsymbol{1} & \boldsymbol{2} & \boldsymbol{3} & \boldsymbol{4} & \boldsymbol{5} & \boldsymbol{6} & \boldsymbol{7} & \boldsymbol{8} & \boldsymbol{9} & \boldsymbol{\hdots} & \boldsymbol{29} \cr
      \boldsymbol{j\equiv0\mods2} & 1 & 0 & 1 & 0 & 1 & 0 & 1 & 0 & 1 & 0 & \hdots & 0 \cr
      \boldsymbol{j\equiv1\mods2} & 0 & 1 & 0 & 1 & 0 & 1 & 0 & 1 & 0 & 1 & \hdots & 1 \cr
      \boldsymbol{j\equiv0\mods3} & 1 & 0 & 0 & 1 & 0 & 0 & 1 & 0 & 0 & 1 & \hdots & 0 \cr
      \boldsymbol{j\equiv1\mods3} & 0 & 1 & 0 & 0 & 1 & 0 & 0 & 1 & 0 & 0 & \hdots & 0 \cr
      \boldsymbol{j\equiv2\mods3} & 0 & 0 & 1 & 0 & 0 & 1 & 0 & 0 & 1 & 0 & \hdots & 1 \cr
      \boldsymbol{j\equiv0\mods5} & 1 & 0 & 0 & 0 & 0 & 1 & 0 & 0 & 0 & 0 & \hdots & 0 \cr
      \boldsymbol{j\equiv1\mods5} & 0 & 1 & 0 & 0 & 0 & 0 & 1 & 0 & 0 & 0 & \hdots & 0 \cr
      \vdots			  &   &   &   &   &   &   &   &   &   &   & \vdots & \cr
      \boldsymbol{j\equiv4\mods5} & 0 & 0 & 0 & 0 & 1 & 0 & 0 & 0 & 0 & 1 & \hdots & 1
      }.
\]
% If we choose e.g.\ $j=6$, then we find that
% \[
%  \mathcal{M}_{2,3,6}=\begin{pmatrix}
%                       r_{1,0} \\
%                       r_{2,0} \\
%                       r_{3,1}
%                      \end{pmatrix}=\begin{pmatrix}
% 				    1 & 0 & 1 & 0 & 1 & 0 & 1 & 0 & 1 & 0 & \hdots & 0 \\
% 				    1 & 0 & 0 & 1 & 0 & 0 & 1 & 0 & 0 & 1 & \hdots & 0 \\
% 				    0 & 1 & 0 & 0 & 0 & 0 & 1 & 0 & 0 & 0 & \hdots & 0
% 				   \end{pmatrix}
% \]
% and 
% \[
%  \mathcal{M}'_{2,3,6}=\begin{pmatrix}
%                        1 & 0 & 1 & 0 & 1 & 0 & 0 & 1 & 0 & \hdots & 0 \\
%                        1 & 0 & 0 & 1 & 0 & 0 & 0 & 0 & 1 & \hdots & 0 \\
%                        0 & 1 & 0 & 0 & 0 & 0 & 0 & 0 & 0 & \hdots & 0
%                       \end{pmatrix}.
% \]
\end{ex}
\end{diss}
\begin{rem}[Restriction of $\widehat{\A_q}$]
If we compute the row-wise Hadamard product of $\HHH$ with $(\widehat{\A_q})^T\in\CC^{1\times q}$, every row of $\HHH\cast(\widehat{\A_q})^T$ is, by Lemma \ref{lem:tensorproduct}, given as the row-wise Hadamard product of a row of $\HHH$ and  $(\widehat{\A_q})^T$. Thus every column $\bfrho^T_{u_m,\nu}$ of $(\HHH\cast(\widehat{\A_q})^T)^T\in\CC^{q\times\mu}$ corresponds to a residue $\nu$ modulo a prime $u_m$. This column only contains nonzero entries at frequencies that are congruent to $\nu$ modulo $u_m$,
 \begin{align*}
%   &\left(\HHH\cast\left(\widehat{\A_q}\right)^T\right)^T_{j,\rho_{u_m,\nu}} =\left(\HHH\cast\left(\widehat{\A_q}\right)^T\right)_{\rho_{u_m,\nu},j} \\
    &\bfrho^T_{u_m,\nu}(j):=\left(\rr_{u_m,\nu}\cast(\widehat{\A_q})^T\right)^T_j=(\rr_{u_m,\nu})^T_j\cdot(\widehat{\A_q})_j \\
   =&\delta\left((j-\nu)\mods u_m\right)\cdot\widehat{\A_q}(j) 
   =\begin{cases}
     \widehat{\A_q}(j), \quad &j\equiv\nu\mods u_m, \\
     0, \quad &\text{otherwise.}
    \end{cases}
 \end{align*}
  This means that the column $\bfrho_{u_m,\nu}^T$ of $(\HHH\cast(\widehat{\A_q})^T)^T$ is the restriction of $\widehat{\A_q}$ to frequencies congruent to $\nu$ modulo $u_m$, which is at most $dn$-sparse for a good hashing prime $u_m$. Thus for more than $M/2$ of the $u_m$ we can apply Algorithm 3 in \cite{Iw-arxiv} with sparsity $dn$ column by column.
\end{rem}
\subsection{Required Technical Background}
\begin{paper}
 In the following we give a short description of Algorithm 3 in \cite{Iw-arxiv} and summarize some of the results proven therein. Said SFT algorithm reconstructs the energetic frequencies and the corresponding Fourier coefficients of a sparse input function $f\colon [0,2\pi]\rightarrow\CC$ with bandwidth $N$ from the Fourier transforms of vectors consisting of $s_kt_l\ll N$ equispaced samples of $f$, where $s_k$ and $t_l$ are small primes depending on the bandwidth and sparsity of the function. The energetic frequencies are reconstructed from their residues modulo $s_k$ and $t_1,\dotsc,t_L$ with the help of the CRT, which is why the primes have to satisfy 
 \[
  \prod_{l=1}^{L-1}t_l<\frac{N}{s_1}\leq\prod_{l=1}^Lt_l.
 \]
 For a general $d$-sparse input function the method introduced in \cite{Iw,Iw-arxiv} is not guaranteed to work for any prime $s_k$. However, setting $K=8d\lfloor\log_{s_1}N\rfloor+1$ and choosing $s_1,\dotsc,s_K$ as the $K$ smallest primes greater than $d$ and $t_L$, for more than $K/2$ of them all energetic frequencies are correctly reconstructed from their residues. These can be found by comparing the entries of $\widehat{\A_{s_k}}$ and $\widehat{\A_{s_kt_l}}$ that correspond to the same frequency for all $l$. The coefficient estimates are then obtained by taking the medians over the $K$ coefficient estimates found for the $s_k$.
\end{paper}
\begin{diss}
In the following we give a short repetition of the SFT algorithm presented in \cite{Iw-arxiv}. Assume that the input function $f$ with bandwidth $N$ is at most $d$ sparse and that there exist pairwise relatively prime numbers $t_1<\dotsb<t_L<s_1<\dotsb<s_K$ with
\[
 \prod_{l=1}^Lt_l\geq\frac{N}{s_1}.
\]
Analogously to $\Mu$ we define the $\kappa\times N$ matrix $\Ms$ and the $(\lambda-1)\times N$ matrix $\Mt$ consisting of the rows corresponding to the possible residues modulo all the $s_k$ and $t_l$, respectively,
\[
 \Ms:=\begin{pmatrix}
       \rr_{s_1,0} \\
       \vdots \\
       \rr_{s_K,s_K-1}
      \end{pmatrix} 
 \quad\text{and}\quad
  \Mt:=\begin{pmatrix}
       \rr_{t_1,0} \\
       \vdots \\
       \rr_{t_L,t_L-1}
      \end{pmatrix}.
\]
We set $\Nt$ to be the $\lambda\times N$ matrix whose first row contains only ones and whose other rows are given by $\Mt$,
\[
 \Nt:=\begin{pmatrix}
       \boldsymbol{1}_N \\
       \Mt
      \end{pmatrix}.
\]
Additionally, we introduce the $(\kappa\cdot\lambda)\times N$ row tensor product $\RRR:=\Ms\cast\Nt$ of $\Ms$ and $\Nt$ and its extension $\GGG$ to a $q\times N$ matrix, where $q:=\lcm(N,s_1,\dotsc,s_K,t_1,\dotsc,t_L)$.
\todo[inline]{Add a short description of what the algorithm does. Maybe the one from the talk?}
\begin{algorithm}
\label{alg:sparsefourierapprox}
\caption{Sparse Fourier Approximation}
\begin{algorithmic}[1]
\renewcommand{\algorithmicrequire}{\textbf{Input:}}
\renewcommand{\algorithmicensure}{\textbf{Output:}}
\Require Function $f+\eta$, $d,N\in\NN$, $K=4\cdot 2d\lfloor\log_{s_1}N\rfloor+1$ and pairwise relatively prime $s_1<\dotsb<s_K,t_1<\dotsb<t_L$ with $t_L<s_1$ and $\prod_{l=1}^Lt_l\geq\frac{N}{s_1}$
\Ensure 
\State Initialize $R=\emptyset$, $\x_R=\mathbf{0}_N$, $q=\lcm(N,s_1,\dotsc,s_K,t_1,\dotsc,t_L)$.
\State $\GGG\cdot\widehat{\A_q}\gets(\widehat{\A_{s_1}}^T,\widehat{\A_{s_1t_1}}^T,\dotsc,\widehat{\A_{s_1t_L}}^T,\widehat{\A_{s_2}}^T,\widehat{\A_{s_2t_1}}^T,\dotsc,\widehat{\A_{s_Kt_L}}^T)^T$ \label{line:dft_sparse}
\State $\EEE\cdot\widehat{\A_q}\gets(\widehat{\A_{s_1}}^T,\widehat{\A_{s_2}}^T,\dotsc,\widehat{\A_{s_K}}^T)^T\text{ (the } \kappa \text{ entries of } \GGG\cdot\widehat{\A_q}$ that approximate $\Ms\cdot\cc(N)$)
\Statex \begin{center} \textsc{Identification of the Energetic Frequencies} \end{center}
\For{$k$ from 1 to $K$} \label{line:identification_start_sparse}
%\State $(1,v,w)\gets\text{extended\_gcd}(s_k,u_m)$ \Comment{i.e. $1=v\cdot s_k+w\cdot u_m$}
\For{$h$ from 0 to $s_k-1$} \label{line:fixed_res}
%\State $h'=\left((h-\nu)w\mods s_k\right)\cdot u_m+\nu$
\For{$l$ from 1 to $L$}
\State $b_{\text{min}}\gets\underset{b\in\{0,\dotsc,{t}_l-1\}}{\argmin}\left|\left(\EEE\cdot\widehat{\A_q}\right)_{r_{s_k,h}}-\left(\GGG\cdot\widehat{\A_q}\right)_{r_{s_kt_l,h+bs_k}}\right|$
\State $a_{l,k,h}\gets(h+b_{\text{min}}s_k)\mods t_l$ \label{line:argmin_sparse}
\EndFor
\State Reconstruct $\omega_{k,h}$ from $\omega_{k,h}\equiv h\mods s_k,\omega_{k,h}\equiv a_{1,k,h}\mods t_l, \;l\in\{1,\dotsc,L\}$\label{line:reconstruction_sparse}
\EndFor
\EndFor \label{line:identification_end_sparse}
\Statex \begin{center} \textsc{Fourier Coefficient Estimation} \end{center}
\For{\textbf{each} $\omega$ reconstructed more than $\frac{K}{2}$ times}
\State $\re\left(x_{\omega}\right)\gets\underset{\substack{k\in\{1,\dotsc,K\} \\ l\in\{0,\dotsc,L\}}}{\median}\left\{\re\left(\left(\GGG\cdot\widehat{\A_q}\right)_{r_{s_kt_l,\omega\mods s_kt_l}}\right)\right\}$ \label{line:re_sparse}
\State $\im\left(x_{\omega}\right)\gets\underset{\substack{k\in\{1,\dotsc,K\} \\ l\in\{0,\dotsc,L\}}}{\median}\left\{\im\left(\left(\GGG\cdot\widehat{\A_q}\right)_{r_{s_kt_l,\omega\mods s_kt_l}}\right)\right\}$ \label{line:im_sparse}
\EndFor \label{line:freq_end_sparse}
\State Sort the nonzero entries by magnitude s.t. $\left|x_{\omega_1}\right|\geq\left|x_{\omega_2}\right|\geq\dotsb$.
\State $R\gets\{\omega_1,\omega_2,\dotsc,\omega_{2d}\}$ \label{line:R_m_sparse}

\State Output: $(R,\x_R)$
\end{algorithmic}
\end{algorithm}
\end{diss}
The following remark summarizes the main results for Algorithm 3 in \cite{Iw-arxiv} that are relevant for this paper.
\begin{rem}\label{rem:sparsesummary}
 Let $f\colon[0,2\pi]\to\CC$ and $d<N\in\NN$.
 \begin{enumerate}
  \item (Lemma 5 in \cite{Iw-arxiv}) By the CRT, every row of $\GGG$ is of the form 
 \[
  \left(\rr_{s_kt_l,h}\right)_j=\left(\rr_{s_k,h\mods s_k}\cast \rr_{t_l,h\mods t_l}\right)_j=\begin{cases}
                                                                                         1, & \text{if } j\equiv h\mods s_kt_l \\
                                                                                         0, & \text{otherwise}
                                                                                        \end{cases},
 \]
 for $h\in\{0,\dotsc,s_kt_l-1\}$, $k\in\{1,\dotsc,K\}$ and $l\in\{0,\dotsc,L\}$, where $t_0:=1$ in order to have the same notation also for rows of the form  $\rr_{s_k,h\mods s_k}\cast \boldsymbol{1}_N$.
  \item (Lemma 6 in \cite{Iw-arxiv}) If $\omega\in\intvl$ is such that 
  \[
   |c_\omega|>4\cdot\left(\frac{1}{2d}\left\|\cc(N)-\cc_{2d}^{\text{opt}}(N)\right\|_1+\left\|\cc(f)-\cc(N,\ZZ)\right\|_1\right),
  \]
  then $\omega$ will be reconstructed more than $\frac{K}{2}$ times. \label{item:sparsesumm_addfreq}
  \item (Proof of Theorem 7 in \cite{Iw-arxiv}) If $\omega$ %\in\intvl$ 
  is reconstructed more than $\frac{K}{2}$ times, then
  \[
   |x_\omega-c_\omega|
   \leq\sqrt{2}\left(\frac{1}{2d}\left\|\cc(N)-\cc_{2d}^{\text{opt}}(N)\right\|_1+\left\|\cc(f)-\cc(N,\ZZ)\right\|_1\right).
  \] \label{item:sparsesum_approx}

  \item (Theorem 7 in \cite{Iw-arxiv}) Algorithm 3 in \cite{Iw-arxiv} will output an $\x_R\in\CC^N$ satisfying
  \begin{align*}
   &\|\cc(N)-\x_R\|_2 \\
   \leq &\left\|\cc(N)-\cc_d^{\text{opt}}(N)\right\|_2+\frac{11}{\sqrt{d}}\left\|\cc(N)-\cc_{2d}^{\text{opt}}(N)\right\|_1
   +22\sqrt{d}\left\|\cc(f)-\cc(N,\ZZ)\right\|_1
  \end{align*}
  in a runtime of 
  \[
   \mathcal{O}\left(\frac{d^2\log^2N\log(d\log N)\log^2\frac{N}{d}}{\log^2d\log\log\frac{N}{d}}\right).
  \]
\begin{diss}
  Lines \ref{line:identification_start_sparse} to \ref{line:freq_end_sparse} have a runtime of 
  \[
   \mathcal{O}\left(\sum_{k=1}^Ks_k\log s_k+Kdn\sum_{l=0}^Lt_l+Kdn\log(Kdn)+Kdn\log K+dn\log dn\right).
  \]
\end{diss}
 \end{enumerate}
\end{rem}
\subsection{Application to Polynomially Structured Sparse Functions} \label{ch:poly_sparse}
We can now apply Algorithm 3 in \cite{Iw-arxiv} with sparsity $dn$ to the columns $\bfrho_{u_m,\nu}^T$ of $(\HHH\cast(\widehat{\A_q})^T)^T$. Recall that $\bfrho_{u_m,\nu}^T$ is only guaranteed to be at most $dn$-sparse if $u_m$ hashes all support sets $S_1,\dotsc,S_n$ well. This means that only for columns corresponding to those primes the algorithm will return all energetic frequencies and good estimates for their Fourier coefficients. Hence we have to apply the algorithm to every single column of $(\HHH\cast(\widehat{\A_q})^T)^T$ and choose those frequencies and coefficient estimates that appear for the more than $\frac{M}{2}$ well-hashing $u_m$. 
\begin{diss}
If we want to apply Algorithm \ref{alg:sparsefourierapprox} to the columns of $\HHH\cast(\widehat{\A_q})^T$, we need to define $\GGG$ using the pairwise relatively prime numbers $t_1<\dotsb<t_L<s_1<\dotsb<s_K$ and our possible hashing primes $B<u_1<\dotsc<u_M$ with 
\[
 \prod_{l=1}^Lt_l\geq\frac{N}{s_1u_1} \quad\text{and}\quad q=\lcm(N,s_1,\dotsc,s_K,t_1,\dotsc,t_L,u_1,\dotsc,u_M)
\]
from Definition \ref{dfn:notation2}. In this case, if we apply the algorithm to the column corresponding to frequencies with the residue $\nu$ modulo $u_m$, we also want to use that $\omega_{k,h}\equiv\nu\mods u_m$ for the reconstruction in line \ref{line:reconstruction_sparse}.
\end{diss}

In Algorithm 3 in \cite{Iw-arxiv}, estimates for the Fourier coefficients $c_\omega$ of the input function $f$ are calculated from certain entries of $\GGG\cdot\widehat{\A_q}$. These entries can be obtained in a fast way from $f$ by computing DFTs of the vectors $\A_{s_kt_l}$.
\begin{rem}
For polynomially structured sparse input functions we now have to show that the entries of $\GGG\cdot(\HHH\cast(\widehat{\A_q})^T)^T$ can also be calculated fast. As we want to use the residues modulo the $u_m$ for the reconstruction as well, an idea similar to the one from \cite{Iw-arxiv} leads to DFTs of the $s_kt_lu_m$-length sample vectors $\A_{s_kt_lu_m}$ from Definition \ref{dfn:notation2}.
\begin{diss}
 ,
 \[
  \A_{s_kt_lu_m}=\left(f\left(\frac{2\pi j}{s_kt_lu_m}\right)\right)_{j=0}^{s_kt_lu_m-1},
 \] 
 where $t_0:=1$ in order to deal with rows all rows of $\GGG$ using the same notation.
\end{diss}
Consider an entry of $\GGG\cdot(\HHH\cast(\widehat{\A_q})^T)^T\in\CC^{\kappa\lambda\times\mu}$ that is given as the product of a row of $\GGG$, by Remark \ref{rem:sparsesummary} of the form $\rr_{s_kt_l,h}$ for a residue $h$ modulo $s_kt_l$, with a column $\bfrho^T_{u_m,\nu}$ of $(\HHH\cast(\widehat{\A_q})^T)^T$ for a residue $\nu$ modulo $u_m$,
%  Every entry of 
%  $\GGG\cdot(\HHH\cast(\widehat{\A_q})^T)^T\in\CC^{\kappa\lambda\times\mu}$ is given as the scalar product of a row of 
%  $\GGG$ with a column of $(\HHH\cast(\widehat{\A_q})^T)^T$. By Remark 
%  \ref{rem:sparsesummary} we know that every row of $\GGG$ is of the form $r_{s_kt_l,h}$ for some residue $h\in\{0,\dotsc,s_kt_l-1\}$, 
%  $l\in\{0,\dotsc,L\}$ and $k\in\{1,\dotsc,K\}$. 
%  Consider the following entry 
%  corresponding to a row $r_{s_kt_l,h}$ of $\GGG$ and a column $\rho^T_{u_m,\nu}$:
 \begin{align}
  \rr_{s_kt_l,h}\cdot\bfrho^T_{u_m,\nu} 
  =&\sum_{j=0}^{q-1} \rr_{s_kt_l,h}(j)\bfrho^T_{u_m,\nu}(j) \notag \\
  =&\sum_{j=0}^{q-1}\delta\left((j-h)\mods s_kt_l\right)\delta\left((j-\nu)\mods u_m\right)\cdot\widehat{\A_q}(j). \label{eq:modulo}
 \end{align}
  As there can only be nonzero summands in (\ref{eq:modulo}) if $j\equiv h\mods s_kt_l$ and $j\equiv \nu\mods u_m$, we find with the CRT that $j$ has to be of the form 
 \[
  j=\tau+j's_kt_lu_m, \quad j'\in\left\{0,\dotsc,\frac{q}{s_kt_lu_m}-1\right\}.
 \]
 Then,
 \begin{align*}
  \rr_{s_kt_l,h}\cdot\bfrho^T_{u_m,\nu}
  =&\sum_{j'=0}^{\frac{q}{s_kt_lu_m}-1}\widehat{\A_q}(\tau+j'\cdot s_kt_lu_m) \notag \\
  =&\sum_{j'=0}^{\frac{q}{s_kt_lu_m}-1}\frac{1}{q}\sum_{b=0}^{q-1}\A_q(b)e^{\frac{-2\pi ib(\tau+j's_kt_lu_m)}{q}} \notag \\
  =&\sum_{b=0}^{q-1}\frac{1}{q}\A_q(b)e^{\frac{-2\pi ib\tau}{q}}\sum_{j'=0}^{\frac{q}{s_kt_lu_m}-1}e^{\frac{-2\pi ibj's_kt_lu_m}{q}} \notag \\
  =&\sum_{b=0}^{q-1}\frac{1}{s_kt_lu_m}f\left(\frac{2\pi b}{q}\right)e^{\frac{-2\pi ib\tau}{q}}\cdot\delta\left(b\mods\frac{q}{s_kt_lu_m}\right) \notag \\ %\displaybreak \\
  =&\sum_{b'=0}^{s_kt_lu_m-1}\frac{1}{s_kt_lu_m}f\left(\frac{2\pi b'\frac{q}{s_kt_lu_m}}{q}\right)e^{\frac{-2\pi i\tau b'\frac{q}{s_kt_lu_m}}{q}} \notag \\
  %=&\sum_{b'=0}^{s_kt_lu_m-1}\frac{1}{s_kt_lu_m}f\left(\frac{2\pi b'}{s_kt_lu_m}\right)e^{\frac{-2\pi ib'\tau}{s_kt_lu_m}} 
  =&\widehat{\A_{s_kt_lu_m}}(\tau). \notag
 \end{align*}
 By Bézout's identity, $1=\gcd(s_kt_l,u_m)=v\cdot s_kt_l+w\cdot u_m$ for some $v,w\in\ZZ$, and we obtain that $\tau$ satisfies
 \begin{equation}\label{eq:tau}
  \tau=\left((h-\nu)w\mods s_kt_l\right)\cdot u_m+\nu\in\{0,\dotsc,s_kt_lu_m-1\}.
 \end{equation}
 Thus we find that, in the column for the residue $\nu$ modulo $u_m$, for fixed $s_kt_l$ only the $s_kt_l$ different values $\widehat{\A_{s_kt_lu_m}}(\tau)$ with $\tau$ depending on $h$ as in (\ref{eq:tau}) are contained. Hence the column of $\GGG\cdot(\HHH\cast(\widehat{\A_q})^T)^T$ corresponding to $\nu$ modulo $u_m$ is of the form 
 \[
  \left(\widehat{\B^{m,\nu}_{s_1}}^T,\widehat{\B^{m,\nu}_{s_1t_1}}^T,\dotsc,\widehat{\B^{m,\nu}_{s_1t_L}}^T, 
\widehat{\B^{m,\nu}_{s_2}}^T,\dotsc,\widehat{\B^{m,\nu}_{s_Kt_L}}^T\right)^T,
 \]
 where 
 \begin{equation}\label{eq:bhat}
  \widehat{\B^{m,\nu}_{s_kt_l}}(j):=\widehat{\A_{s_kt_lu_m}}\left((j-\nu)w\mods s_kt_l)\cdot u_m+\nu\right) 
  \quad \forall j\in\{0,\dotsc,s_kt_l-1\}.
 \end{equation}
\begin{diss}
 Recall that 
 \[
  A_n(j)=f\left(\frac{2\pi j}{n}\right);
 \]
 so all that is required to compute the entries of $\GGG\cdot(\HHH\cast(\widehat{\A_q})^T)^T$ is the evaluation of the input function $f$ at points of the form $\frac{2\pi j}{s_kt_lu_m}$ for all $j$, $k$, $l$ and $m$.

 The entries of $\GGG\cdot(\HHH\cast(\widehat{\A_q})^T)^T$ can be calculated in a fast way, using $KLM$ DFTs of length $s_k t_l u_m$ with runtime $\mathcal{O}(s_kt_lu_m\cdot\log(s_kt_lu_m))$ for all $k,l,m$. 

 We have to bear in mind that, due to the fact that not all of the $u_m$ hash all frequency sets well, we have to apply Algorithm \ref{alg:sparsefourierapprox} to all columns $\rho_{u_m,\nu}^T$ and choose the frequencies that occur for more than $\frac{M}{2}$ of the possible hashing primes. 
\end{diss}
\begin{paper}
The entries of $\GGG\cdot(\HHH\cast(\widehat{\A_q})^T)^T$ can be calculated in a fast way, using $KLM$ DFTs of vectors of $s_kt_l u_m$ equispaced samples with runtime $\mathcal{O}(s_kt_lu_m\cdot\log(s_kt_lu_m))$ for all $k,l,m$. 
\end{paper}

 How exactly do we apply Algorithm 3 in \cite{Iw-arxiv} to the columns of $\GGG\cdot(\HHH\cast(\widehat{\A_q})^T)^T$? Until now we considered a fixed residue $h$ modulo $s_kt_l$ and a fixed residue $\nu$ modulo $u_m$. However, in line 7 of that algorithm we fix the residue $h'$ modulo $s_k$ of a frequency and find its residues modulo the $s_kt_l$ in line 9. In the setting of polynomially structured sparse functions this means that for a frequency $\omega$ with residue $\nu$ modulo $u_m$ and residue $h'$ modulo $s_k$ we have to find the corresponding residue modulo $s_kt_lu_m$ for every $l$. Then
 %For a frequency $\omega$ that satisfies $\omega\equiv\nu\mods u_m$ and $\omega\equiv h\mods s_k$, we obtain that
 \[
  \tau':=\omega\mods s_ku_m=((h'-\nu)w\mods s_k)\cdot u_m+\nu 
 \]
 holds for the residue of $\omega$ modulo $s_ku_m$, where Bézout's identity implies that
 \[
  1=\gcd(s_k,u_m)=v'\cdot s_k+w'\cdot u_m
 \]
 for some $v',w'\in\ZZ$. Then the residue of $\omega$ modulo $s_kt_lu_m$ is of the form
 \[
  \omega\mods s_kt_lu_m= \tau'+b_{\text{min}}\cdot s_ku_m 
 \]
 for a $b_{\text{min}}\in\{0,\dotsc,t_l-1\}$, which is given as
 \begin{equation}\label{eq:bhatstuff}
  b_{\text{min}}:=\underset{b\in\{0,\dotsc,t_l-1\}}{\argmin}\left|\widehat{\A_{s_ku_m}}(\tau')-\widehat{\A_{s_kt_lu_m}}(\tau'+b\cdot s_ku_m)\right|.
 \end{equation}
 Finally, the residue of $\omega$ modulo $t_l$ is 
 \[
  a_{l}:=\omega\mods t_l=(\tau'+b_{\text{min}}\cdot s_ku_m)\mods t_l,
 \]
 and $\omega$ can be reconstructed from its residues $\omega\equiv \tau'\mods s_ku_m$, $\omega\equiv a_{1}\mods t_1$,$\dotsc$,$\omega\equiv a_{L}\mods t_L$. Recall the notion of the $\widehat{\B_{s_kt_l}\mn}$ introduced in (\ref{eq:bhat}). These vectors are defined such that if $\omega\equiv\nu\mods u_m$ and $\omega\equiv h\mods s_kt_l$, we have
 \[
  \widehat{\B_{s_kt_l}^{m,\nu}}(\omega\mods s_kt_l)=\widehat{\A_{s_kt_lu_m}}(\omega\mods s_kt_lu_m).
 \]
 To use this notation, we take the residues modulo $s_kt_lu_m$ again modulo $s_kt_l$, and obtain the following,
%  which are given by 
%  \[
%   f_b:=(g+b\cdot s_ku_m) \mods s_kt_l \quad \forall b\in\{0,\dotsc,t_l-1\},
%  \]
%  since $g+b\cdot s_ku_m$ is the residue modulo $s_kt_lu_m$. Consequently, we obtain for (\ref{eq:bhatstuff})
 \begin{align*}
  b_{\text{min}}=&\underset{b\in\{0,\dotsc,t_l-1\}}{\argmin}\left|\widehat{\A_{s_ku_m}}(\tau')-\widehat{\A_{s_kt_lu_m}}(\tau'+b\cdot s_ku_m)\right| \\
  =&\underset{b\in\{0,\dotsc,t_l-1\}}{\argmin}\left|\widehat{\B_{s_k}^{m,\nu}}(h')
  -\widehat{\B_{s_kt_l}^{m,\nu}}((\tau'+b\cdot s_ku_m)\mods s_kt_l)\right| \\
  =&\underset{b\in\{0,\dotsc,t_l-1\}}{\argmin}\left|\left(\EEE\cdot\widehat{\A_q}\right)_{\rr_{s_k,h'},\bfrho_{m,\nu}^T}
  -\left(\GGG\cdot\widehat{\A_q}\right)_{\rr_{s_kt_l,(\tau'+bs_ku_m)\mods s_kt_l},\bfrho_{m,\nu}^T}\right|.
 \end{align*}
\end{rem}
 Algorithm \ref{alg:fourierapprox} presents itself as a summary of the preceding considerations.
 \begin{algorithm}
\caption{Fourier Approximation}
\label{alg:fourierapprox}
\begin{algorithmic}[1]
\renewcommand{\algorithmicrequire}{\textbf{Input:}}
\renewcommand{\algorithmicensure}{\textbf{Output:}}
\Require Function $f+\eta$, $K=8dn\lfloor\log_{s_1}\frac{N}{u_1}\rfloor+1$, 
$M=2(n+1)\cdot\lfloor\log_{u_1}N\rfloor+1$ and pairwise relatively prime $s_1<\dotsb<s_K,t_1<\dotsb<t_L,B<u_1<\dotsb<u_M$ with $t_L<s_1$, $u_m$ prime and $\prod_{l=1}^Lt_l\geq\frac{N}{s_1u_1}$.
\Ensure $R,\x_R$, where $R$ contains the $nB$ frequencies $\omega$ with greatest magnitude coefficient estimates $x_R(\omega)$.
\State Initialize $R=\emptyset$, $\x_R=\mathbf{0}_N$, $q=\lcm(N,s_1,\dotsc,s_K,t_1,\dotsc,t_L,u_1,\dotsc,u_M)$ \label{line:initialization}
\State $\GGG\cdot(\HHH\cast(\widehat{\A_q})^T)^T\gets\biggl(\Bigl(\widehat{\B_{s_1}^{m,\nu}}^T,\widehat{\B_{s_1t_1}^{m,\nu}}^T,\dotsc,\widehat{\B_{s_Kt_L}^{m,\nu}}^T\Bigr)^T\biggr)_{m=1,\nu=0}^{M,u_m-1}$ \label{line:dft}
% \For{$m$ from 1 to $M$} \label{line:rewrite_start}
\State $\EEE\cdot(\HHH\cast(\widehat{\A_q})^T)^T\gets\biggl(\Bigl(\widehat{\B_{s_1}^{m,\nu}}^T,\dotsc,\widehat{\B_{s_K}^{m,\nu}}^T\Bigr)\biggr)_{m=1,\nu=0}^{M,u_m-1}$  
% (the $\kappa\mu$ entries of  $\GGG\cdot(\HHH\cast(\widehat{\A_q})^T)^T$ that approximate $\Ms\cdot(\Mu\cast(\cc(N))^T)^T$ 
% \EndFor \label{line:rewrite_end}
%\Statex \begin{center} \textsc{Identification of the Energetic Frequencies} \end{center}
\For{$m$ from 1 to $M$} \label{line:freq_start}
\For{$\nu$ from 0 to $u_m-1$} \label{line:identification_start}
% \For{$k$ from 1 to $K$}
%\State $(1,v,w)\gets\text{extended\_gcd}(s_k,u_m)$ \Comment{i.e. $1=v\cdot s_k+w\cdot u_m$}
\State $(R^{(m,\nu)},\x^{(m,\nu)})\gets$ $2dn$ frequencies with largest magnitude coefficient estimates returned by Algorithm 3 in \cite{Iw-arxiv} applied to $\GGG\cdot(\HHH\cast(\widehat{\A_q})^T)^T_{\bfrho_{m,\nu}^T}$ and $\EEE\cdot(\HHH\cast(\widehat{\A_q})^T)^T_{\bfrho_{m,\nu}^T}$ with sparsity $dn$. \label{line:apply_sft}
% \State Sort the nonzero entries by magnitude s.t. $\left|x_{\omega_1}^{(m,\nu)}\right|\geq\left|x_{\omega_2}^{(m,\nu)}\right|\geq\dotsb$
% \State $R^{(m,\nu)}\gets\{\omega_1,\omega_2,\dotsc,\omega_{2dn}\}$ \label{line:R_mnu}
% \EndFor
% \State $R^m\gets\bigcup_{\nu=0}^{u_m-1}R^{(m,\nu)}$ \label{line:R_m}
% \For{$\omega\in R^m$}
% \State $\x^m(\omega)\gets x^{(m,\nu)}_\omega \quad\text{if } \omega\equiv\nu\mods u_m$ \label{line:x_m}
% \EndFor
\EndFor 
\EndFor \label{line:freq_end}
\For{\textbf{each} $\omega\in\bigcup_{m=1}^M\bigcup_{\nu=0}^{u_m-1}R^{(m,\nu)}$ found more than $\frac{M}{2}$ times} \label{line:approx_start}
\State $\re(x_\omega)=\underset{\subalign{ \nu&=0,\dotsc,u_m-1 \\ m&=1,\dotsc,M}}{\median} \left\{\re\left(x_{\tilde\omega}^{(m,\nu)}\right):\tilde\omega=\omega,\tilde\omega\in R\mn\right\}$ \label{line:coeffapprox_re}
\State $\im(x_\omega)=\underset{\subalign{ \nu&=0,\dotsc,u_m-1 \\ m&=1,\dotsc,M}}{\median}\left\{\im\left(x_{\tilde\omega}^{(m,\nu)}\right)\colon \tilde\omega=\omega,\tilde\omega\in R\mn\right\}$ \label{line:coeffapprox_im}
% \State $x(\omega) =\underset{\substack{\tilde\omega\in \cup_{m,\nu}R^{(m,\nu)} \\ \tilde\omega=\omega}}{\median}\left\{\re\left(x_{\tilde\omega}^{(m,\nu)}\right)\right\}+ i\cdot\underset{\substack{\tilde\omega\in \cup_{m,\nu}R^{(m,\nu)} \\ \tilde\omega=\omega}}{\median}\left\{\im\left(x_{\tilde\omega}^{(m,\nu)}\right)\right\}$ \label{line:coeffapprox}
\EndFor \label{line:approx_end}
\State Sort the coefficients by magnitude s.t. $|x_{\omega_1}|\geq|x_{\omega_2}|\geq\dotsb$ \label{line:sort_est}
\State Output: $R=\{\omega_1,\dotsc,\omega_{nB}\},\x_R$ \label{line:addfreq2}
\end{algorithmic}
\end{algorithm}
\subsection{Results for Polynomially Structured Sparse Functions}\label{ch:general_results}
In order to obtain bounds on the accuracy and runtime of our algorithm, we can utilize some of the results developed in \cite{Iw-arxiv}, at least for the more than $\frac{M}{2}$ primes $u_m$ that hash all support sets $S_1,\dotsc,S_n$ well, i.e., the primes where the columns $\bfrho_{u_m,\nu}^T$ of $(\HHH\cast(\widehat{\A_q})^T)^T$ are guaranteed to be at most $dn$-sparse. 

Analogously to our previous notation we denote by $\cc(N,u_m,\nu)$, $\cc(N,\ZZ,u_m,\nu)$ and $\cc(u_m,\nu)$ the restrictions of $\cc(N)$, $\cc(N,\ZZ)$ and $\cc(f+\eta)$, respectively, to the frequencies congruent to $\nu$ modulo $u_m$. Further, recall that $\cc^{\text{opt}}_{2dn}(N,u_m,\nu)$ is the optimal $2dn$-term representation of $\cc(N,u_m,\nu)$.

The following lemma guarantees that all significantly enough frequencies will be found and their Fourier coefficients estimated well.
\begin{lem}\label{lem:approx}
 Let $N\in\NN$ and $f\colon[0,2\pi]\rightarrow\CC$ be $P(n,d,B)$-structured sparse with noise $\eta$ such that $\cc(\eta)\in\ell^1$ and $\|\cc(\eta)\|_\infty\leq\eps$. Let $u_1$ be a prime and $s_1$, $t_1$ natural numbers such that with $K=8dn\lfloor\log_{s_1}\frac{N}{u_1}\rfloor+1$ and $M=2(n+1)\lfloor\log_{u_1}N\rfloor+1$ we have that $B<u_1<\dotsb<u_M$ are primes, $t_1<\dotsb<t_L<s_1<\dotsb<s_K$, and $u_1,\dotsc,u_M,s_1,\dotsc,s_K,t_1,\dotsc,t_L$ are pairwise relatively prime with $\prod_{l=1}^Lt_l\geq\frac{N}{s_1u_1}$. Set 
 \begin{align*}
  \delta=&\max_{\substack{\nu=0,\dotsc,u_m-1 \\ u_m\,\text{hashes}\,\text{well}}} \left\{\delta^{(m,\nu)}\right\} \\
  =&\max_{\substack{\nu=0,\dotsc,u_m-1 \\u_m\,\text{hashes}\,\text{well}}} \left\{\frac{1}{2dn}\left\|\cc(N,u_m,\nu)
  -\cc^{\text{opt}}_{2dn}(N,u_m,\nu)\right\|_1
  +\left\|\cc(N,\ZZ,u_m,\nu)-\cc(u_m,\nu)\right\|_1\right\}.
 \end{align*}
 Then each $\omega\in\intvl$ with $|c_\omega|>\eps+4\delta$ is added to the output $R$ of Algorithm \ref{alg:fourierapprox} in line \ref{line:addfreq2}, and its coefficient estimate from lines \ref{line:coeffapprox_re} and \ref{line:coeffapprox_im} satisfies
 \[
  |x_R(\omega)-c_\omega|\leq2\delta.
 \]
\end{lem}
\begin{proof}
 Let $u_m$ be a good hashing prime and assume that $\omega\in\intvl$ is contained in $R^{(m,\nu),\text{opt}}_{dn}\backslash R^{(m,\nu)}$, i.e., that it is one of the $dn$ largest magnitude Fourier coefficient frequencies that are congruent to $\nu$ modulo $u_m$, but not contained in the set of frequencies returned by Algorithm 3 in \cite{Iw-arxiv} applied to $\GGG\cdot(\HHH\cast(\widehat{\A_q})^T)^T_{\bfrho_{m,\nu}^T}$ for sparsity $dn$. 
 
 If $|c_\omega|\leq\eps+4\delta$, then $|c_\omega|$ is small enough that not including it in the reconstruction $R$ does not yield large errors, since $\delta$ is defined as the maximum over the 
 \[
 \delta\mn=\frac{1}{2dn}\left\|\cc(N,u_m,\nu)-\cc^{\text{opt}}_{2dn}(N,u_m,\nu)\right\|_1+\left\|\cc(N,\ZZ,u_m,\nu)-\cc(u_m,\nu)\right\|_1
 \]
 for all good hashing primes $u_m$. 
\begin{diss}
If $u_m$ hashes all support sets $S_1,\dotsc,S_n$ well, the restriction of $\cc(N)$ to the frequencies congruent to $\nu$ modulo $u_m$ and the optimal $dn$-term representation of the same vector cannot deviate much, since there are at most $dn$ energetic frequencies with any given residue, so 
 \[
 \frac{1}{2dn}\left\|\cc(N,u_m,\nu)-\cc^{\text{opt}}_{2dn}(N,u_m,\nu)\right\|_1
 \]
 is small. Further, the deviation between the embedding of the restriction of the embedding of $\cc(N)$ into $\ZZ$ to the frequencies congruent to $\nu$ modulo $u_m$ and the restriction of $\cc$ to these frequencies is really small, as we suppose that no significantly large frequencies of $f$ lie outside of $\intvl$. Thus not including a frequency $\omega$ with $|c_\omega|\leq4\delta$ does not have a large impact on the reconstruction accuracy.
\end{diss}
 
 If $|c_\omega|>\eps+4\cdot\delta\geq 4\cdot\delta^{(m,\nu)}$ for all good hashing primes $u_m$, we know by Remark \ref{rem:sparsesummary} \ref{item:sparsesumm_addfreq} that $\omega$ will be reconstructed more than $\frac{K}{2}$ times by Algorithm 3 in \cite{Iw-arxiv}. Then $\omega$ can only not be contained in $R^{(m,\nu)}$ if there are $dn+1$ frequencies $\tilde\omega$ in $R^{(m,\nu)}\backslash R^{(m,\nu),\text{opt}}_{dn}$ that satisfy $\left|x_{\tilde\omega}^{(m,\nu)}\right|\geq\left|x_{\omega}^{(m,\nu)}\right|$. Recall that $u_m$ hashes all support sets $S_1,\dotsc,S_n$ well, so there are at most $dn$ energetic frequencies congruent to $\nu$ modulo $u_m$. Suppose that all frequencies with this residue are ordered by magnitude of their Fourier coefficient, i.e.,
 \[
  \left|c_{\omega_1\mn}\right|\geq\left|c_{\omega_2\mn}\right|\geq\dotsb\geq\left|c_{\omega_{dn}\mn}\right|
  \geq\underbrace{\left|c_{\omega_{dn+1}\mn}\right|}_{\leq\eps}\geq\dotsb\quad.
 \]
 Then $|c_{\tilde\omega}|\leq\left|c_{\omega_{dn+1}\mn}\right|\leq|c_\omega|$ for all $\tilde\omega$. By Remark \ref{rem:sparsesummary} \ref{item:sparsesum_approx} we have for all $\bar\omega$ that were reconstructed more than $\frac{K}{2}$ times by Algorithm 3 in \cite{Iw-arxiv} for $\nu$ modulo $u_m$ that
 \begin{equation}\label{eq:someapprox}
  \left|x_{\bar\omega}^{(m,\nu)}-c_{\bar\omega}\right|\leq\sqrt{2}\delta^{(m,\nu)},
 \end{equation}
 and further 
\begin{diss}
 \begin{align}
  \left|-x_{\bar\omega}^{(m,\nu)}\right|-\left|c_{\bar\omega}\right|
  \leq\left|c_{\bar\omega}-x_{\bar\omega}^{(m,\nu)}\right|\leq\sqrt{2}\delta\mn
  \quad\Leftrightarrow\quad &\left|x_{\bar\omega}^{(m,\nu)}\right|\leq\left|c_{\bar\omega}\right|+\sqrt{2}\delta\mn, \label{eq:ineq1} \\
  \left|c_{\bar\omega}\right|-\left|-x_{\bar\omega}^{(m,\nu)}\right|
  \leq\left|c_{\bar\omega}-x_{\bar\omega}^{(m,\nu)}\right|\leq\sqrt{2}\delta\mn
  \quad\Leftrightarrow\quad &\left|x_{\bar\omega}^{(m,\nu)}\right|\geq\left|c_{\bar\omega}\right|-\sqrt{2}\delta\mn. \label{eq:ineq2}
 \end{align}
\end{diss}
\begin{paper}
 \begin{align}
  &\left|x_{\bar\omega}^{(m,\nu)}\right|\leq\left|c_{\bar\omega}\right|+\sqrt{2}\delta\mn \label{eq:ineq1} \\
  &\left|x_{\bar\omega}^{(m,\nu)}\right|\geq\left|c_{\bar\omega}\right|-\sqrt{2}\delta\mn. \label{eq:ineq2}
 \end{align}
\end{paper}
 (\ref{eq:someapprox}) - (\ref{eq:ineq2}) hold for the $\omega$ and $\tilde\omega$ from above, and we find for all $\omega\in R^{(m,\nu),\text{opt}}_{dn}\backslash R\mn$ that
 \begin{align*}
  &\left|c_{\omega\mn_{dn+1}}\right|+\sqrt{2}\delta\mn\geq\left|c_{\tilde\omega}\right|+\sqrt{2}\delta\mn
  \geq\left|x_{\tilde\omega}\mn\right| \\ 
  \geq& |x_\omega| \geq\left|c_{\omega}\right|-\sqrt{2}\delta\mn\geq\left|c_{\omega_{dn+1}\mn}\right|-\sqrt{2}\delta\mn.
 \end{align*}
 Then
 \[
  \left|c_{\omega}\right|\leq\left|c_{\omega_{dn+1}\mn}\right|+2\sqrt{2}\delta\mn\leq\eps+2\sqrt{2}\delta\mn,
 \]
 which contradicts $|c_\omega|>\eps+4\delta$, so we know that $\omega\in R\mn$. Since this holds for all more than $\frac{M}{2}$ good hashing primes, $\omega$ will be considered from line \ref{line:approx_start} onward. We now prove the accuracy of the coefficient estimate. From (\ref{eq:someapprox}) it follows that 
\[
 \left|\re\left(x^{(m,\nu)}_\omega\right)-\re(c_\omega)\right|
 \leq \left|x_{\omega}^{(m,\nu)}-c_{\omega}\right|\leq\sqrt{2}\delta^{(m,\nu)} \leq \sqrt{2}\delta,
\]
and analogously for the imaginary parts. As the estimates hold for more than $\frac{M}{2}$ of the hashing primes $u_m$, they also hold for the medians in lines \ref{line:coeffapprox_re} and \ref{line:coeffapprox_im}. These are taken over the at most $M$ coefficient estimates $x_{\tilde\omega}\mn$ for $\omega$, since for each $u_m$ an $\tilde\omega=\omega$ can appear in at most one set $R\mn$. Hence, we obtain 
\[
 \left|\re\left(x_\omega\right)-\re(c_\omega)\right| \leq \sqrt{2}\delta \quad\text{and}\quad
 \left|\im\left(x_\omega\right)-\im(c_\omega)\right| \leq \sqrt{2}\delta,
\]
and finally
\[
  \left|x_\omega-c_\omega\right|=\sqrt{\left(\re\left(x_\omega-c_\omega\right)\right)^2
 +\left(\im\left(x_\omega-c_\omega\right)\right)^2} \notag 
 \leq\sqrt{\left(\sqrt{2}\delta\right)^2+\left(\sqrt{2}\delta\right)^2}=2\delta.
\]
All that remains to be shown is that $\omega$ will actually be added to $R$ in line \ref{line:addfreq2}. Similarly to (\ref{eq:ineq2}) we find that $|x_\omega|\geq|c_\omega|-2\delta$. Together with $|c_\omega|>\eps+4\delta$ this implies that $|x_\omega|>\eps+2\delta$. Then $\omega$ is only not included in the output if $x_\omega$ is not among the $Bn$ largest magnitude coefficient estimates, i.e., if there exist $Bn$ other frequencies $\tilde\omega$ that satisfy $|x_{\tilde\omega}|\geq|x_\omega|$. We know that $\omega$ is energetic, which means that at least one of these $\tilde\omega$ must have a Fourier coefficient with $|c_{\tilde\omega}|\leq\eps$. Then an analogue to (\ref{eq:ineq1}) yields
 \[
  |x_\omega|\leq|x_{\tilde\omega}|\leq|c_{\tilde\omega}|+2\delta\leq\eps+2\delta,
 \]
 which contradicts $|x_\omega|>\eps+2\delta$. Hence, $\omega$
\begin{diss} 
  is really among the $nB$ frequencies with the largest magnitude coefficient estimates and 
\end{diss}
 will be added to $R$ in line \ref{line:addfreq2}.
\end{proof}
\begin{diss}
This lemma tells us that if $\omega\in R^{(m,\nu),\text{opt}}_{dn}\backslash R\mn$, then $|c_\omega|\leq\eps+4\delta$. If $|c_\omega|>\eps+4\delta$, it will always be contained in $R$, so if $\omega$ is energetic enough, it will be reconstructed and its Fourier coefficient will be estimated up to a $2\delta$-accuracy.
\end{diss}
\begin{rem}\label{rem:estimates}
 One way to ensure that the requirements of Algorithm \ref{alg:fourierapprox} are met is to define $t_1,\dotsc,t_L$ as the $L$ smallest primes satisfying
 \[
  \prod_{l=1}^{L-1}t_l<\frac{N}{Bdn}\leq\prod_{l=1}^Lt_l.
 \]
 Set $s_1$ as the smallest prime that is greater than both $dn$ and $t_L$,
 \[
  s_1:=p_x>\max\{dn,t_L\}\geq p_{x-1}.
 \]
 Instead of taking the minimal $K$, we can increase it slightly by using that $u_1>B$, i.e.,
 \[
  K=8dn\left\lfloor\log_{s_1}\frac{N}{B}\right\rfloor+1\geq8dn\left\lfloor\log_{s_1}\frac{N}{u_1}\right\rfloor+1.
 \]
 Hence, we can now choose the remaining $s_k$ independently from the $u_m$ to be $s_k:=p_{x-1+k}$ for $k\in\{1,\dotsc,K\}$. The hashing primes $u_m$ can then be found by setting 
 \[
  u_1:=p_y>\max\{B,s_K\}\geq p_{y-1},
 \]
 $M=2(n+1)\lfloor\log_{u_1} N\rfloor+1$ and $u_m:=p_{y-1+m}$ for $m\in\{1,\dotsc,M\}$. 
 With these definitions $t_1,\dotsc,t_L,s_1,\dotsc,s_K,u_1,\dotsc,u_M$ are pairwise relatively prime and satisfy that 
 \[
  s_ku_m\cdot\prod_{l=1}^Lt_l\geq N ,
 \]
 as well as $s_k>dn$ and $u_m>B$ for all $k$ and $m$.
\end{rem}
Using the $s_k$, $t_l$ and $u_m$ from Remark \ref{rem:estimates}, the following main theorem gives us the runtime and error bounds of Algorithm \ref{alg:fourierapprox}.
\begin{thm}\label{thm:main}
 Let $N\in\NN$ and $f\colon[0,2\pi]\rightarrow\CC$ be $P(n,d,B)$-structured sparse with noise $\eta$ such that $\cc(\eta)\in\ell^1$ and $\|\cc(\eta)\|_\infty\leq\eps$. Let $t_1,\dotsc,t_L$ be the smallest primes with $\prod_{l=1}^Lt_l\geq\frac{N}{Bdn}$. Set $s_1$ as the smallest prime greater than $\max\{dn,t_L\}$, $K=8dn\lfloor\log_{s_1}\frac{N}{B}\rfloor+1$ and $s_2,\dotsc,s_K$ as the first $K-1$ primes greater than $s_1$. Let $u_1$ be the smallest prime greater than $\max\{B,s_K\}$, $M=2(n+1)\lfloor\log_{u_1}N\rfloor+1$ and $u_2,\dotsc,u_M$ the first $M-1$ primes greater than $u_1$. Let further $\delta$ be defined as 
 \begin{align*}
  \delta %=&\max_{\substack{\nu=0,\dotsc,u_m-1 \\ u_m \text{ hashes well}}} \left\{\delta^{(m,\nu)}\right\} \\
  :=&\max_{\substack{\nu=0,\dotsc,u_m-1 \\u_m\,\mathrm{hashes}\,\mathrm{well}}} \left\{\frac{1}{2dn}\|\cc(N,u_m,\nu)
  -\cc\opt_{2dn}(N,u_m,\nu)\|_1
  +\left\|\cc(N,\ZZ,u_m,\nu)-\cc(u_m,\nu)\right\|_1\right\}.
 \end{align*} 
 Then the output $(R,\x_R)$ of Algorithm \ref{alg:fourierapprox} satisfies
 \[
  \|\cc(N)-\x_R\|_2\leq\left\|\cc(N)-\cc\opt_{Bn}(N)\right\|_2+\sqrt{Bn}\cdot(\eps+6\delta).
 \]
 If $B>s_K$, the output can be computed in a runtime of
 \[
  \mathcal{O}\left(\frac{d^2n^3(B+n\log N)\cdot\log^2\frac{N}{Bdn}\log^2\frac{N}{B}\log N\log\left(dn\log \frac{N}{B}\right)
  \log^2\left(\frac{B+n\log N}{\log B}\right)}{\log^2B\log^2(dn)\log\log\frac{N}{Bdn}}\right)
 \]
 and the algorithm has a sampling complexity of
 \[
  \mathcal{O}\left(\frac{d^2n^3(B+n\log N)\cdot\log^2\frac{N}{Bdn}\log^2\frac{N}{B}\log N\log\left(dn\log \frac{N}{B}\right)
  \log\left(\frac{B+n\log N}{\log B}\right)}{\log^2B\log^2(dn)\log\log\frac{N}{Bdn}}\right).
 \]
 % B\leq s_K only in the dissertation.
 \begin{diss}
 If $B\leq s_K$, the algorithm has a runtime of
 \[
  \mathcal{O}\left(\frac{n^4d^3\log^2\frac{N}{Bdn}\log^2\frac{N}{B}\log^2N\log\left(dn\log\frac{N}{B}\right)\log^2(dn\log N)}
  {\log^3(dn)\log B\log\log \frac{N}{Bdn}}\right)
 \]
 with a sampling complexity of
 \[
  \mathcal{O}\left(\frac{n^4d^3\log^2\frac{N}{Bdn}\log^2\frac{N}{B}\log^2N\log\left(dn\log\frac{N}{B}\right)\log(dn\log N)}
  {\log^3(dn)\log B\log\log \frac{N}{Bdn}}\right).
 \]
\end{diss}
\end{thm}
\begin{proof}
 For the vector $\cc_R(N)$, whose entries are the Fourier coefficients $c_\omega$ for the frequencies contained in $R$ and zero otherwise, it always holds that
 \begin{equation}\label{eq:first_est}
  \|\cc(N)-\x_R\|_2\leq\|\cc(N)-\cc_{R}(N)\|_2+\|\cc_R(N)-\x_R\|_2.
 \end{equation}
  The square of first summand in (\ref{eq:first_est}) can be written as
 \begin{align*}
  &\|\cc(N)-\cc_R(N)\|_2^2=\sum_{\omega=\intu}^{\into}|c_\omega-c_R(\omega)|^2 \\
  =&\sum_{\omega\notin R}|c_\omega|^2+\sum_{\omega\in R}\underbrace{|c_\omega-c_R(\omega)|^2}_{=0}
  =\sum_{\omega\notin R^{\text{opt}}_{Bn}}|c_\omega|^2+\sum_{\omega\in R_{Bn}^{\text{opt}}\backslash R}|c_\omega|^2
  -\sum_{\omega\in R\backslash R_{Bn}^{\text{opt}}}|c_\omega|^2 \\
  =&\left\|\cc(N)-\cc_{Bn}^{\text{opt}}(N)\right\|_2^2+\sum_{\omega\in R_{Bn}^{\text{opt}}\backslash R}|c_\omega|^2
  -\sum_{\omega\in R\backslash R_{Bn}^{\text{opt}}}|c_\omega|^2.
 \end{align*}
 For every $\omega\in R^{\text{opt}}_{Bn}\backslash R$ we know by Lemma \ref{lem:approx} that $|c_\omega|\leq\eps+4\delta$, because otherwise it would be contained in $R$. Since $R^{\text{opt}}_{Bn}\backslash R$ contains at most $Bn$ elements, this yields
 \begin{align*}
  &\|\cc(N)-\cc_R(N)\|_2^2
  =\left\|\cc(N)-\cc_{Bn}^{\text{opt}}(N)\right\|_2^2+\underbrace{\sum_{\omega\in R_{Bn}^{\text{opt}}\backslash R}|c_\omega|^2}_{\leq Bn(\eps+4\delta)^2}
  -\underbrace{\sum_{\omega\in R\backslash R_{Bn}^{\text{opt}}}|c_\omega|^2}_{\geq0} \\
  \leq& \left\|\cc(N)-\cc_{Bn}^{\text{opt}}(N)\right\|_2^2+Bn(\eps+4\delta)^2.
 \end{align*}
For the second summand in (\ref{eq:first_est}) consider an $\omega\in R$. For each of the more than $\frac{M}{2}$ 
good hashing primes it has to be contained in exactly one of the $R\mn$, so $\omega$ must have been reconstructed more than $\frac{K}{2}$ times by Algorithm 3 in \cite{Iw-arxiv}, applied to the entries congruent to $\nu\equiv\omega\mods u_m$. Hence we have by (\ref{eq:someapprox}) that
 \[
  \left|x_\omega\mn-c_\omega\right|\leq\sqrt2\delta\mn \quad\text{and}\quad |x_R(\omega)-c_\omega|\leq2\delta,
 \]
 analogously to the proof of Lemma \ref{lem:approx}. Since $R$ contains at most $Bn$ elements, we find
 \[
  \|\cc_R(N)-\x_R\|_2^2=\sum_{\omega\in R}\underbrace{|c_\omega-x_R(\omega)|^2}_{\leq 4\delta^2} %\leq\sum_{\omega\in R}4\delta^2
  \leq4Bn\delta^2.
 \]
 Combining all these estimates we obtain that
 \begin{align*}
  \|\cc(N)-\x_R\|_2 \leq &\sqrt{\|\cc(N)-\cc_R(N)\|_2^2+\|\cc_R(N)-\x_R\|_2^2} \\
  \leq &\left\|\cc(N)-\cc_{Bn}^{\text{opt}}(N)\right\|_2+\sqrt{Bn}(\eps+4\delta)+2\sqrt{Bn}\delta \\
  =&\left\|\cc(N)-\cc_{Bn}^{\text{opt}}(N)\right\|_2+\sqrt{Bn}(\eps+6\delta).
 \end{align*}
 In order to determine the runtime of the algorithm let us first consider the runtime of the calculation of the DFTs in line \ref{line:dft}. It was shown in \cite{Iw-arxiv, iwenspencer} that
 \[
  t_L=\mathcal{O}\left(\log\frac{N}{Bdn}\right) \quad\text{and} \quad
  s_K=\mathcal{O}\left(dn\log_{dn}\frac{N}{B}\log\left(dn\log\frac{N}{B}\right)\right).
 \]
 Let $\pi$ be the prime-counting function,
 \[
  \pi(x)=\sum_{\substack{2\leq p\leq x \\ p \text{ prime}}}1.
 \]
 If $s_K\leq B$, we can set $u_1:=p_y>B\geq p_{y-1}$ to be the first prime greater than $B$, so by the Prime Number Theorem (see \cite{montgomery})
 \[
  y-1=\pi(B)=\mathcal{O}\left(\frac{B}{\log B}\right)
 \]
 and
 \[
  y-1+M=\mathcal{O}\left(\frac{B}{\log B}+n\log_BN\right)=\mathcal{O}\left(\frac{B+n\log N}{\log B}\right).
 \]
 An equivalent formulation of the Prime Number Theorem yields for $u_M=p_{y-1+M}$ that
 \[
  u_M=\mathcal{O}((y-1+M)\log(y-1+M))=\mathcal{O}\left(\frac{B+n\log N}{\log B}\log\left(\frac{B+n\log N}{\log B}\right)\right).
 \]
 In \cite{iwenspencer} it was proven that 
 \[
  \sum_{\substack{2\leq p\leq R \\ p \text{ prime}}}p=\mathcal{O}\left(\frac{R^2}{\log R}\right) .
  %\quad \text{and}\quad  \sum_{\substack{p\leq R \\ p \text{ prime}}}p\log p=\mathcal{O}\left(R^2\right).
 \]
 Because estimating $\sum_{m=1}^Mu_m\log u_m$ by summing $p\log p$ for all primes less than $u_M$ would take into account many primes that do not contribute to the sum, we use instead that
 \[
  \sum_{m=1}^Mu_m\log u_m=\mathcal{O}(M\cdot u_M\log u_M).
 \]
\begin{diss}
 Recall that 
 \[
  \widehat{\B^{m,\nu}_{s_kt_lu_m}}(j)=\widehat{\A_{s_kt_lu_m}}\left((j-\nu)w\mods s_kt_l)\cdot u_m+\nu\right) 
  \quad \forall j\in\{0,\dotsc,s_kt_l-1\}.
 \]
\end{diss}
 In line \ref{line:dft} we have to calculate the DFTs of length $s_kt_lu_m$ of the vectors $\A_{s_kt_lu_m}$ for all $k,l,m$. Even if some $\tilde m\in\NN$ is not a power of 2, computing a DFT of length $\tilde m$ has a runtime of $\mathcal{O}(\tilde m\log \tilde m)$ (see \cite{bluestein,rabiner}). Since $u_1>s_K$, we obtain a computational effort of
\begin{diss}
 \begin{align*}
  &\mathcal{O}\left(\sum_{k=1}^K\sum_{l=0}^L\sum_{m=1}^Ms_kt_lu_m\log(s_kt_lu_m)\right) 
  =\mathcal{O}\left(\sum_{k=1}^Ks_k\sum_{l=0}^Lt_l\sum_{m=1}^Mu_m\log u_m\right) \\
  =&\mathcal{O}\left(\frac{t_L^2}{\log t_L}\cdot \frac{s_K^2}{\log s_K}\cdot M\cdot u_M\log u_M\right) \\  
  =&\mathcal{O}\left(\frac{\log^2\frac{N}{Bdn}}{\log\log\frac{N}{Bdn}}
  \cdot\frac{(dn)^2\log^2\frac{N}{B}\log^2\left(dn\log \frac{N}{B}\right)}
  {\log^2(dn)\log\left(dn\log_{dn}\frac{N}{B}\log\left(dn\log\frac{N}{B}\right)\right)} \cdot n\log_BN\right.\\
  \cdot &\left.\left(\frac{B+n\log N}{\log B}\right)\log\left(\frac{B+n\log N}{\log B}\right)
  \log\left(\left(\frac{B+n\log N}{\log B}\right)\log\left(\frac{B+n\log N}{\log B}\right)\right)\right) \\
  =&\mathcal{O}\left(\frac{d^2n^3(B+n\log N)\cdot\log^2\frac{N}{Bdn}\log^2\frac{N}{B}\log N\log(dn\log \frac{N}{B})
  \log^2\left(\frac{B+n\log N}{\log B}\right)}{\log^2B\log^2(dn)\log\log\frac{N}{Bdn}}\right).  
  %=&\mathcal{O}\left(\frac{(dn)^4\log^4N\log^2\frac{N}{Bdn}\log^4(dn\log N)}{\log^4B\log\log\frac{N}{Bdn}\log(dn\log N)}\right)
 \end{align*}
\end{diss}
\begin{paper}
 \begin{align*}
  &\mathcal{O}\left(\sum_{k=1}^K\sum_{l=0}^L\sum_{m=1}^Ms_kt_lu_m\log(s_kt_lu_m)\right) 
  =\mathcal{O}\left(\sum_{k=1}^Ks_k\sum_{l=0}^Lt_l\sum_{m=1}^Mu_m\log u_m\right) \\
  =&\mathcal{O}\left(\frac{t_L^2}{\log t_L}\cdot \frac{s_K^2}{\log s_K}\cdot M\cdot u_M\log u_M\right) \\  
  =&\mathcal{O}\left(\frac{d^2n^3(B+n\log N)\cdot\log^2\frac{N}{Bdn}\log^2\frac{N}{B}\log N\log(dn\log \frac{N}{B})
  \log^2\left(\frac{B+n\log N}{\log B}\right)}{\log^2B\log^2(dn)\log\log\frac{N}{Bdn}}\right).  
  %=&\mathcal{O}\left(\frac{(dn)^4\log^4N\log^2\frac{N}{Bdn}\log^4(dn\log N)}{\log^4B\log\log\frac{N}{Bdn}\log(dn\log N)}\right)
 \end{align*}
\end{paper}
 For the sampling complexity we find
\begin{diss}
 \begin{align*}
  &\mathcal{O}\left(\sum_{k=1}^Ks_k\sum_{l=0}^Lt_l\sum_{m=1}^Mu_m\right) 
  =\mathcal{O}\left(\frac{t_L^2}{\log t_L}\cdot \frac{s_K^2}{\log s_K}\cdot M\cdot u_M\right) \\  
  =&\mathcal{O}\left(\frac{(dn)^2\log_{dn}^2\frac{N}{B}\log^2\left(dn\log \frac{N}{B}\right) \cdot \log^2\frac{N}{Bdn}
  \cdot n\log_BN\cdot\left(\frac{B+n\log N}{\log B}\right)\log\left(\frac{B+n\log N}{\log B}\right) }
  {\log\left(dn\log_{dn}\frac{N}{B}\log\left(dn\log\frac{N}{B}\right)\right)\cdot \log\log\frac{N}{Bdn}} \right) \\
  =&\mathcal{O}\left(\frac{d^2n^3(B+n\log N)\cdot\log^2\frac{N}{Bdn}\log^2\frac{N}{B}\log N\log(dn\log \frac{N}{B})
  \log\left(\frac{B+n\log N}{\log B}\right)}{\log^2B\log^2(dn)\log\log\frac{N}{Bdn}}\right).  
 \end{align*}
\end{diss}
\begin{paper}
 \begin{align*}
  &\mathcal{O}\left(\sum_{k=1}^Ks_k\sum_{l=0}^Lt_l\sum_{m=1}^Mu_m\right) 
  =\mathcal{O}\left(\frac{t_L^2}{\log t_L}\cdot \frac{s_K^2}{\log s_K}\cdot M\cdot u_M\right) \\  
  =&\mathcal{O}\left(\frac{d^2n^3(B+n\log N)\cdot\log^2\frac{N}{Bdn}\log^2\frac{N}{B}\log N\log(dn\log \frac{N}{B})
  \log\left(\frac{B+n\log N}{\log B}\right)}{\log^2B\log^2(dn)\log\log\frac{N}{Bdn}}\right).  
 \end{align*}
\end{paper}
 % B\leq s_K case only in ther dissertation.
 \begin{diss}
 If $\pi(B)<x-1+K$, then $s_K\geq B$, so we can set the $u_m$ as the first $M$ primes greater than $s_K$, i.e.\ $y=x+K$. Then 
 \begin{align*}
  &x-1+K+M \\
  =&8dn\left\lfloor\log_{s_1}\frac{N}{B}\right\rfloor+1+2(n+1)\left\lfloor\log_{u_1}N\right\rfloor+1
  +\mathcal{O}\left(\frac{dn}{\log dn}+\frac{\log\frac{N}{Bdn}}{\log\log\frac{N}{Bdn}}\right) \\
  =&\mathcal{O}\left(dn\log_{dn}N+\frac{dn}{\log dn}+\frac{\log\frac{N}{Bdn}}{\log\log\frac{N}{Bdn}}\right)
 \end{align*}
 and thus
 \begin{align*}
  u_M=\mathcal{O}\left((x-1+K+M)\log(x-1+K+M)\right)=\mathcal{O}\left(dn\log_{dn}N\log(dn\log N)\right).
 \end{align*}
 Combining all these estimates yields that the runtime of the computation of the DFTs in line \ref{line:dft} is
 \begin{align*}
  &\mathcal{O}\left(\sum_{k=1}^K\sum_{l=0}^L\sum_{m=1}^Ms_kt_lu_m\log(s_kt_lu_m)\right) \\
  =&\mathcal{O}\left(\frac{\log^2\frac{N}{Bdn}}{\log\log\frac{N}{Bdn}}
  \cdot \frac{(dn)^2\log^2\frac{N}{B}\log^2\left(dn\log\frac{N}{B}\right)}
  {\log^2(dn)\log\left(dn\log_{dn}\frac{N}{B}\log\left(dn\log \frac{N}{B}\right)\right)}\right. \\
  &\left.\cdot n\log_BN dn\log_{dn}N\log(dn\log N)\log\left(dn\log_{dn}N\log\left(dn\log N\right)\right)\right) \\
  =&\mathcal{O}\left(\frac{n^4d^3\log^2\frac{N}{Bdn}\log^2\frac{N}{B}\log^2N\log\left(dn\log\frac{N}{B}\right)\log^2(dn\log N)}
  {\log^3(dn)\log B\log\log \frac{N}{Bdn}}\right)
 \end{align*}
 and that the sampling complexity is 
 \begin{align*}
  &\mathcal{O}\left(\sum_{k=1}^Ks_k\sum_{l=0}^Lt_l\sum_{m=1}^Mu_m\right) 
  =\mathcal{O}\left(\frac{t_L^2}{\log t_L}\cdot \frac{s_K^2}{\log s_K}\cdot M\cdot u_M\right) \\  
  =&\mathcal{O}\left(\frac{\log^2\frac{N}{Bdn}}{\log\log\frac{N}{Bdn}}
  \cdot \frac{(dn)^2\log^2\frac{N}{B}\log^2\left(dn\log\frac{N}{B}\right)
  \cdot n\log_BN\cdot dn\log_{dn}N\log(dn\log N)}
  {\log^2(dn)\log\left(dn\log_{dn}\frac{N}{B}\log\left(dn\log \frac{N}{B}\right)\right)}\right) \\
  =&\mathcal{O}\left(\frac{n^4d^3\log^2\frac{N}{Bdn}\log^2\frac{N}{B}\log^2N\log\left(dn\log\frac{N}{B}\right)\log(dn\log N)}
  {\log^3(dn)\log B\log\log \frac{N}{Bdn}}\right).
 \end{align*}
\end{diss}
Now we can estimate the runtime of the remaining steps of the algorithm. The $q$ in line \ref{line:initialization} does not actually have to be computed, it is just defined there in order to introduce more readable notation.
% Since the $s_k$, $t_l$ and $u_m$ are pairwise relatively prime, we know that
% \[
%  \lcm(N,s_1,\dotsc,s_K,t_1,\dotsc,t_L,u_1,\dotsc,u_M)=\lcm\left(N,\prod_{k=1}^K\prod_{l=1}^L\prod_{m=1}^Ms_kt_lu_m\right).
% \]
% Further, it holds for any two integers $a>b$ that
% \[
%  \lcm(a,b)=\frac{a\cdot b}{\gcd(a,b)},
% \]
% and $\gcd(a,b)$ can be computed in $\mathcal{O}(\log b)$ time (see \cite{cormen}). Hence, line \ref{line:initialization} has a runtime of 
% \[
%  \mathcal{O}(K+L+M+\log(N))=\mathcal{O}(dn\log N).
% \]
% \begin{diss}
% since $N\geq\prod_{l=1}^Lt_ls_1u_1$. 
% \end{diss}
In line \ref{line:apply_sft} we apply Algorithm 3 in \cite{Iw-arxiv} to the column corresponding to the residue $\nu$ modulo $u_m$ of $\GGG\cdot(\HHH\cast(\widehat{\A_q})^T)^T$, which was already computed in line \ref{line:dft}. We know from Remark \ref{rem:sparsesummary} that the runtime of Algorithm 3 in \cite{Iw-arxiv}
\begin{diss}
without the computation of the DFTs in line \ref{line:dft_sparse} is 
\begin{align*}
 &\mathcal{O}\left(\sum_{k=1}^Ks_k\log s_k+Kdn\sum_{l=0}^Lt_l+Kdn\log(Kdn)+Kdn\log K+dn\log dn\right) \\
 =&\mathcal{O}\left(\sum_{k=1}^K\sum_{l=0}^Ls_kt_l\log s_k+Kdn\log K+Kdn\log dn\right) \\
 =&\mathcal{O}\left(\sum_{k=1}^K\sum_{l=0}^Ls_kt_l\log s_k+(dn)^2\log^2N\right) \\
 =&\mathcal{O}\left(\sum_{k=1}^K\sum_{l=0}^Ls_kt_l\log s_k\right),
\end{align*}
so the runtime of these lines is actually much smaller than the runtime of line \ref{line:dft_sparse}, since the estimates used above are very rough. Consequently, we find that
\end{diss}
\begin{paper}
is dominated by the computation of the DFTs, so
\end{paper}
the runtime of lines \ref{line:freq_start} to \ref{line:freq_end} of Algorithm \ref{alg:fourierapprox} is dominated by the runtime of line \ref{line:dft}.
\begin{diss}
\[
 \mathcal{O}\left(\sum_{m=1}^M\sum_{\nu=0}^{u_m-1}\sum_{k=1}^K\sum_{l=0}^Lt_ls_k\log s_k\right)
 =\mathcal{O}\left(\sum_{k=1}^K\sum_{l=0}^L\sum_{m=1}^Ms_kt_lu_m\log s_k\right).
\]
\end{diss}
In order to find out for which frequencies line \ref{line:approx_start} to \ref{line:approx_end} have to be executed, we can sort the $2dn\sum_{m=1}^Mu_m$ frequencies that are returned by all the calls of Algorithm 3 in \cite{Iw-arxiv} by size and count how often each distinct frequency appears. This can be done in 
\begin{diss}
\begin{align*}
 &\mathcal{O}\left(2dn\left(\sum_{m=1}^Mu_m\right)\cdot\log\left(2dn\sum_{m=1}^Mu_m\right)\right) \\
 =&\mathcal{O}\left(dn\cdot M\cdot u_M\log(dn\cdot M\cdot u_M)\right)
 =\mathcal{O}\left(\frac{s_K^2}{\log s_K}\cdot\frac{t_L^2}{\log t_L}\cdot u_M^2\right)
\end{align*}
\end{diss}
\begin{paper}
\begin{align*}
 &\mathcal{O}\left(2dn\left(\sum_{m=1}^Mu_m\right)\cdot\log\left(2dn\sum_{m=1}^Mu_m\right)\right) 
 =\mathcal{O}\left(dn M u_M\cdot\log(dn M u_M)\right)
\end{align*}
\end{paper}
time, so it is also insignificant compared to the DFT computation. There are at most 
\begin{diss}
\[
 \frac{2}{M}\cdot\sum_{m=1}^M\sum_{\nu=0}^{u_m-1}2dn=\frac{4dn}{M}\cdot\sum_{m=1}^Mu_m
 =\mathcal{O}\left(\frac{4dnMu_M}{M}\right)=\mathcal{O}\left(4dn\cdot u_M\right)
\]
\end{diss}
\begin{paper}
\[
 \frac{2}{M}\cdot\sum_{m=1}^M\sum_{\nu=0}^{u_m-1}2dn=\frac{4dn}{M}\cdot\sum_{m=1}^Mu_m
 =\mathcal{O}\left(4dn\cdot u_M\right)
\]
\end{paper}
frequencies that can have been found more than $\frac{M}{2}$ times. If we fix one of these frequencies, $\omega$, then for each $u_m$ there is exactly one residue $\nu^{(m)}\in\{0,\dotsc,u_m-1\}$ with $\omega\equiv\nu^{(m)}\mods u_m$. Since the $2dn$ frequencies recovered for any fixed residue modulo some hashing prime are distinct, there can be at most $M$ frequencies $\tilde\omega$ satisfying $\tilde\omega=\omega$ found for all hashing primes. This means that the medians in lines \ref{line:coeffapprox_re} and \ref{line:coeffapprox_im} are taken over at most $M$ elements. As medians can be computed by sorting, both lines have a runtime of $\mathcal{O}(M\log M)$. Combining these considerations we obtain that lines \ref{line:approx_start} to \ref{line:approx_end} require 
\[
 \mathcal{O}\left(4dn\cdot u_M\cdot M\log M\right)
\]
arithmetical operations, which is dominated by the effort of the DFT computations in line \ref{line:dft}. Finally, sorting the $\mathcal{O}(4dn\cdot u_M)$ coefficient estimates in line \ref{line:sort_est} has a runtime of
\[
 \mathcal{O}(4dnu_M\log(4dnu_M)),
\]
so, as stated above, the runtime of Algorithm \ref{alg:fourierapprox} is determined by the one of line \ref{line:dft}.
\end{proof}

If $f+\eta$ is bandlimited, simplifying the above error bound yields Theorem \ref{thm:main_intro_blockSparse} in \S\ref{ch:intro_results}.
\begin{cor}\label{cor:bandlimited_general}
 Let $N\in\NN$ and $f\colon[0,2\pi]\rightarrow\CC$ be $P(n,d,B)$-structured sparse with noise $\eta$ such that $\cc(\eta)\in\ell^1$, $\|\cc(\eta)\|_\infty\leq\eps$ and $f$ and $f+\eta$ are bandlimited to $\intvl$. Choosing the $s_k,t_l,u_m$ as in Theorem \ref{thm:main}, the output $(R,\x_R)$ of Algorithm \ref{alg:fourierapprox} satisfies
 \[
  \|\cc(N)-\x_R\|_2\leq\left\|\cc(N)-\cc\opt_{Bn}(N)\right\|_2+\sqrt{Bn}\left(\eps+\frac{3}{dn}\left\|\cc(N)-\cc\opt_{2Bn}(N)\right\|_1\right).
 \] 
\end{cor}
\begin{proof}
 By definition of $\delta$ we have that 
 \begin{align*}
  \delta:=&\max_{\substack{\nu=0,\dotsc,u_m-1 \\ u_m \text{ hashes well}}} \left\{\delta^{(m,\nu)}\right\} =\delta^{(m',\nu')} \\
  =& \frac{1}{2dn}\left\|\cc(N,u_{m'},\nu')-\cc^{\text{opt}}_{2dn}(N,u_{m'},\nu')\right\|_1
  +\left\|\cc(N,\ZZ,u_{m'},\nu')-\cc(u_{m'},\nu')\right\|_1
 \end{align*} 
 for some residue $\nu'$ modulo a good hashing prime $u_{m'}$. Since $f$ and $f+\eta$ are bandlimited, the second summand is 0. Due to the fact that $f$ is $P(n,d,B)$-structured sparse, any restriction $\cc(N,u_{m'},\nu)$ of $\cc(f)$ to the frequencies that are congruent to $\nu$ modulo $u_{m'}$ is $dn$-sparse. As there are at most $Bn$ energetic frequencies, we find the following estimate,
 \begin{align*}
  \delta&\leq\sum_{\nu=0}^{u_{m'}-1}\delta^{(m',\nu)} 
  \leq \sum_{\nu=0}^{u_{m'}-1}\frac{1}{2dn}\left\|\cc(N,u_{m'},\nu)-\cc^{\text{opt}}_{2dn}(N,u_{m'},\nu)\right\|_1 \\
  &\leq \frac{1}{2dn}\left\|\cc(N)-\cc^{\text{opt}}_{2Bn}(N)\right\|_1,
 \end{align*}
 so the error bound from Theorem \ref{thm:main} reduces to
 \begin{align*}
  \|\cc(N)-\x_R\|_2\leq &\left\|\cc(N)-\cc\opt_{Bn}(N)\right\|_2+\sqrt{Bn}\cdot(\eps+6\delta) \\
  \leq &\left\|\cc(N)-\cc\opt_{Bn}(N)\right\|_2+\sqrt{Bn}\cdot\left(\eps+\frac{3}{dn}\left\|\cc(N)-\cc\opt_{2Bn}(N)\right\|_1\right).
 \end{align*}

\end{proof}

\section{Algorithm for Functions with Simplified Fourier Structure}
\label{sec:Alg_for_Special_Props}

The algorithm introduced in \S\ref{ch:poly_sparse} always uses $M$ hashing primes of which more than $M/2$ are good. If we can guarantee that one hashing prime suffices, a simplified, faster version of Algorithm \ref{alg:fourierapprox} can be applied, which is what we will study in the following. 
\subsection{Structured Sparse Functions Requiring Only One Hashing Prime}
If certain additional information about the polynomials generating the support sets $S_1,\dotsc,S_n$ is known, the number of required hashing primes can be reduced to one. We know by Lemma \ref{lem:wellhashed_iff} that a prime $u$ does not hash a support set $S_j$ well if and only if $u$ divides all non-constant coefficients. Thus we can make the following observation.
\begin{thm}\label{thm:side}
 Let $N\in\NN$ and $f\colon[0,2\pi]\rightarrow\CC$ be $P(n,d,B)$-structured sparse with noise $\eta$ such that $\cc(\eta)\in\ell^1$ and $\|\cc(\eta)\|_\infty\leq\eps$. Let the support set $S=\bigcup_{j=1}^nS_j$ of $f$ be defined by the non-constant polynomials $P_j(x)=\sum_{k=0}^da_{jk}x^k$ for $j\in\{1,\dotsc,n\}$. Let $u>B$ be a prime such that for all $j\in\{1,\dotsc,n\}$ there exists a $k_j\in\{1,\dotsc,d\}$ with $p\nmid a_{jk_j}$. Then $u$ hashes all support sets well. Set $M=1$ and the $s_k$ and $t_l$ as in Theorem \ref{thm:main}. If $B>s_K$, the runtime of Algorithm \ref{alg:fourierapprox} reduces to 
 \[
  \mathcal{O}\left(\frac{u\log u\cdot(dn)^2\log^2\frac{N}{Bdn}\log^2\frac{N}{B}\log\left(dn\log\frac{N}{B}\right)}
  {\log^2(dn)\log\log\frac{N}{Bdn}}\right),
 \]
 while only 
 \[
  \mathcal{O}\left(\frac{u\cdot(dn)^2\log^2\frac{N}{Bdn}\log^2\frac{N}{B}\log\left(dn\log\frac{N}{B}\right)}
  {\log^2(dn)\log\log\frac{N}{Bdn}}\right)
 \]
 samples of $f+\eta$ are being used. If $B\leq s_K$, we obtain a runtime of  
 \[
  \mathcal{O}\left(\frac{u\cdot(dn)^2\cdot\log^2\frac{N}{Bdn}\log^2\frac{N}{B}\log^2\left(dn\log\frac{N}{B}\right)}{\log^2(dn)\log\log\frac{N}{Bdn}}\right)
 \]
 and a sampling complexity of
 \[
  \mathcal{O}\left(\frac{u\cdot (dn)^2\cdot \log^2\frac{N}{Bdn}\log^2\frac{N}{B}\log\left(dn\log\frac{N}{B}\right)}{\log^2(dn)\log\log\frac{N}{Bdn}}\right). 
 \]
\end{thm}
\begin{proof}
 Lemma \ref{lem:wellhashed_iff} implies that $u$ hashes all $n$ support sets well, so the restriction to the frequencies congruent to $\nu$ modulo $u$ is at most $dn$-sparse for all residues. Hence, we can apply Algorithm 3 in \cite{Iw-arxiv} to $\GGG\cdot(\HHH\cast(\widehat{\A_q})^T)^T_{\bfrho_{m,\nu}^T}$ for every residue $\nu$ modulo $u$, and will always obtain a good reconstruction. As there are no residues modulo which more than $dn$ energetic frequencies can collide, it suffices to set $u_1=u$ and $M=1$. Lines \ref{line:approx_start} to \ref{line:approx_end} do not have to be executed, since every frequency will be recovered in line \ref{line:apply_sft} for exactly one residue.
  
 If $u>s_K$, we can define the $t_l$ and $s_k$ as in Remark \ref{rem:estimates}. 
\begin{paper}
 If $u\leq s_K$, then $u$ might collide with one of the $s_k$ or $t_l$. In that case we shift all the $t_l$ and $s_k$, starting by $u$, to the next largest prime, so they are at most the next largest prime greater than the original $t_l$ and $s_k$. This does not change the estimates in the proof of Theorem \ref{thm:main}.
\end{paper}
\begin{diss}
 If $u<s_K$, there are three possible cases, for which we introduce the slightly modified primes $\tilde s_k$ and $\tilde t_l$.
 \begin{enumerate}
  \item $t_L<u<s_1$ \label{item:side_i}
  \item $\exists k'\in\{1,\dotsc,K\}\colon u=s_{k'}$ \label{item:side_ii}
  \item $\exists l'\in\{1,\dotsc,L\}\colon u=t_{l'}$ \label{item:side_iii}
 \end{enumerate}
 In case \ref{item:side_i} we set $\tilde s_k:=s_k$ and $\tilde t_l:=t_l$ for all $k\in\{1,\dotsc,K\}$ and $l\in\{1,\dotsc,L\}$. In case 
 \ref{item:side_ii} we choose instead of the $s_k$ primes $\tilde s_k$ defined by
 \[
  \tilde s_k=\begin{cases}
              s_k, \quad &\text{if } k\in\{1,\dotsc,k'-1\}, \\
              s_{k+1}=p_{x+k}, \quad&\text{if } k\in\{k',\dotsc,K\}.
             \end{cases}
 \]
 Then $\tilde s_K=s_{K+1}:=p_{x+K}$, so we have that $\tilde s_K=\mathcal{O}(s_K)$, which means that the estimates from Remark \ref{rem:estimates} still hold. Analogously, we can introduce the primes $\tilde t_l$ and $\tilde s_k$ in case \ref{item:side_iii}, given by
  \[
  \tilde t_l=\begin{cases}
              t_l, \quad &\text{if } l\in\{1,\dotsc,l'-1\}, \\
              t_{l+1}=p_{l+1}, \quad&\text{if } l\in\{l',\dotsc,L\}.
             \end{cases}
 \]
 and
 \[
  \tilde s_1:=p_{\tilde x}>\max\{dn,\tilde t_L\}\geq p_{\tilde x-1}, \quad \tilde s_k=p_{\tilde x-1+k} \quad 
             \forall k\in\{1,\dotsc,K\}.
 \]
 Then $\tilde t_L=t_{L+1}:=p_{L+1}$ and $\tilde s_1$ is the smallest prime greater than both $dn$ and the first prime greater than $t_L$, so we have that $\tilde t_L=\mathcal{O}(t_L)$ and $\tilde s_K=\mathcal{O}(s_K)$. Thus we can use the same estimates as before. 
\end{diss} 
 
 Let us first consider the case that $u>s_K$. We obtain for the computation of the DFTs in line \ref{line:dft}, which dominates the runtime of Algorithm \ref{alg:fourierapprox}, that they require
\begin{diss}
 \begin{align*}
  &\mathcal{O}\left(\sum_{k=1}^K\sum_{l=0}^L s_k t_l u\log(s_k t_l u)\right)
  =\mathcal{O}\left(\sum_{k=1}^K\sum_{l=0}^Ls_kt_Lu(\log s_k+\log t_l+\log u)\right) \\
  =&\mathcal{O}\left(u\log u\cdot\sum_{l=1}^L t_L\sum_{k=1}^K s_k\right) 
  =\mathcal{O}\left(u\log u\cdot\frac{t_L^2}{\log t_L}\cdot\frac{s_K^2}{\log s_K}\right) \\
  =&\mathcal{O}\left(u\log u\cdot\frac{\log^2\frac{N}{Bdn}}{\log\log\frac{N}{Bdn}}
  \cdot\frac{(dn)^2\log^2\frac{N}{B}\log^2\left(dn\log\frac{N}{B}\right)}
  {\log^2(dn)\log\left(dn\log_{dn}\frac{N}{B}\log\left(dn\log\frac{N}{B}\right)\right)}\right) \\
  =&\mathcal{O}\left(\frac{u\log u\cdot(dn)^2\log^2\frac{N}{Bdn}\log^2\frac{N}{B}\log\left(dn\log\frac{N}{B}\right)}
  {\log^2(dn)\log\log\frac{N}{Bdn}}\right)
 \end{align*}
\end{diss}
\begin{paper}
 \begin{align*}
  &\mathcal{O}\left(\sum_{k=1}^K\sum_{l=0}^L s_k t_l u\log(s_k t_l u)\right) 
  =\mathcal{O}\left(u\log u\cdot\frac{t_L^2}{\log t_L}\cdot\frac{s_K^2}{\log s_K}\right) \\
  =&\mathcal{O}\left(\frac{u\log u\cdot(dn)^2\log^2\frac{N}{Bdn}\log^2\frac{N}{B}\log\left(dn\log\frac{N}{B}\right)}
  {\log^2(dn)\log\log\frac{N}{Bdn}}\right)
 \end{align*}
\end{paper}
 arithmetical operations and have a sampling complexity of
 \begin{align*}
  &\mathcal{O}\left(\sum_{k=1}^K\sum_{l=0}^L s_k t_l u\right)  
  =\mathcal{O}\left(u\cdot\frac{t_L^2}{\log t_L}\cdot \frac{s_K^2}{\log s_K}\right) \\
  =&\mathcal{O}\left(\frac{u\cdot(dn)^2\log^2\frac{N}{Bdn}\log^2\frac{N}{B}\log\left(dn\log\frac{N}{B}\right)}
  {\log^2(dn)\log\log\frac{N}{Bdn}}\right).
 \end{align*}
\begin{diss}
 If $u\leq s_K$, we have to introduce the modified primes $\tilde s_k$ and $\tilde t_l$, with which we obtain a runtime of
 \begin{align*}
  &\mathcal{O}\left(\sum_{k=1}^K\sum_{l=0}^L\tilde s_k\tilde t_l u\log(\tilde s_k\tilde t_l u)\right)
  =\mathcal{O}\left(u\cdot\sum_{l=1}^L\tilde t_l\sum_{k=1}^K\tilde s_k\log \tilde s_k\right) \\
  =&\mathcal{O}\left(u\cdot\frac{t_L^2}{\log t_L}\cdot s_K^2\right)
  =\mathcal{O}\left(\frac{u\cdot(dn)^2\cdot\log^2\frac{N}{Bdn}\log^2\frac{N}{B}\log^2\left(dn\log\frac{N}{B}\right)}
  {\log^2(dn)\log\log\frac{N}{Bdn}}\right)
 \end{align*}
 and a sampling complexity of 
 \begin{align*}
  &\mathcal{O}\left(\sum_{k=1}^K\sum_{l=0}^L\tilde s_k\tilde t_l u\right)  
  =\mathcal{O}\left(u\cdot\frac{t_L^2}{\log t_l}\cdot \frac{s_K^2}{\log s_K}\right) \\
  =&\mathcal{O}\left(\frac{u\cdot(dn)^2\cdot\log^2\frac{N}{Bdn}\log^2\frac{N}{B}\log^2\left(dn\log\frac{N}{B}\right)}
  {\log^2(dn)\log\log\frac{N}{Bdn}\log\left(dn\log_{dn}\frac{N}{B}\log\left(dn\log\frac{N}{B}\right)\right)}\right) \\
  =&\mathcal{O}\left(\frac{u\cdot(dn)^2\cdot\log^2\frac{N}{Bdn}\log^2\frac{N}{B}\log\left(dn\log\frac{N}{B}\right)}
  {\log^2(dn)\log\log\frac{N}{Bdn}}\right).
 \end{align*}
\end{diss}
\begin{paper}
 If $u\leq s_K$, we obtain a runtime of
 \begin{align*}
  &\mathcal{O}\left(u\cdot\sum_{l=0}^L t_l\sum_{k=1}^K s_k\log s_k\right) 
  =\mathcal{O}\left(u\cdot\frac{t_L^2}{\log t_L}\cdot s_K^2\right) \\
  =&\mathcal{O}\left(\frac{u\cdot(dn)^2\cdot\log^2\frac{N}{Bdn}\log^2\frac{N}{B}\log^2\left(dn\log\frac{N}{B}\right)}
  {\log^2(dn)\log\log\frac{N}{Bdn}}\right)
 \end{align*}
 and a sampling complexity of 
 \begin{align*}
  &\mathcal{O}\left(\sum_{k=1}^K\sum_{l=0}^L s_k t_l u\right)  
  =\mathcal{O}\left(u\cdot\frac{t_L^2}{\log t_L}\cdot \frac{s_K^2}{\log s_K}\right) \\
  =&\mathcal{O}\left(\frac{u\cdot(dn)^2\cdot\log^2\frac{N}{Bdn}\log^2\frac{N}{B}\log\left(dn\log\frac{N}{B}\right)}
  {\log^2(dn)\log\log\frac{N}{Bdn}}\right).
 \end{align*}
\end{paper}
 \end{proof}
We now give some conditions on the coefficients of the polynomials $P_1,\dotsc,P_n$ generating the support sets $S_1,\dotsc,S_n$ which guarantee that all $S_j$ are hashed well. All of the conditions arise by tightening the necessary and sufficient requirement of the existence of a non-constant coefficient that is not divisible by $u$ in Theorem \ref{thm:side}. Hence, all of the conditions are sufficient, but they may not be necessary anymore, which might make them easier to prove in practice.
\begin{lem}\label{lem:special_cases_B}
 Let $f$ be $P(n,d,B)$-structured sparse. In the following cases any prime $u>B$ is guaranteed to hash all frequency subsets well,
 \begin{enumerate}[series=specialcases]
  \item $\forall\, j\in\{1,\dotsc,n\}\colon \gcd\left(a_{j1},\dotsc,a_{jd}\right)<B$, which includes 
  $\gcd\left(a_{j_1},\dotsc,a_{jd}\right)=1$, \label{item:special_start}
  %\item $\forall j\in\{1,\dotsc,n\}\, \exists k_j\in\{1,\dotsc,d\} \colon \left|a_{j{k_j}}\right|=\prod_{v=1}^Vp_v^{e_v}$ where $e_v\in\NN_0$ for all $v\in\{1,\dotsc,V\}$ and $p_V\leq B$ (prime factorization) 
  \item $\forall\, j\in\{1,\dotsc,n\}\, \exists\, k_j\in\{1,\dotsc,d\} \colon \left|a_{jk_j}\right|<B$, \label{item:special_mid}
  \item $\forall\, j\in\{1,\dotsc,n\}\, \exists\, k_j\in\{1,\dotsc,d\} \colon  a_{jk_j}=1$, which includes monic polynomials,
  \item $\forall\, j\in\{1,\dotsc,n\} \colon \deg(P_j)=1$ and $a_{j1}=1$, which is the block sparse case. \label{item:block} \label{item:special_end}
  %\item $\forall j\in\{1,\dotsc,n\}\, \exists k\in\{1,\dotsc,d\} \colon |a_{jk_j}|=2^{l_j},$ $k_i\in\NN_0$. 
 \end{enumerate}
\end{lem}
If we have already fixed a prime $u>B$ that is supposed to be the hashing prime, the following conditions imply that $u$ indeed hashes all support sets well.
\begin{lem}\label{lem:special_cases_u}
Let $f$ be $P(n,d,B)$-structured sparse. In the following cases a fixed prime $u>B$ is guaranteed to hash all support sets well.
\begin{enumerate}[resume=specialcases]
 \item \ref{item:special_start} to \ref{item:special_end} from Lemma \ref{lem:special_cases_B} hold for $u$, and $B$ can be changed to $u$ for \ref{item:special_start} and \ref{item:special_mid}
 \item $\forall\, j\in\{1,\dotsc,n\}\colon u\nmid\sum_{k=1}^da_{jk},$
 \item $\forall\, j\in\{1,\dotsc,n\}\,\exists\,\eps_j\in\{0,1\}^d\colon u\nmid \sum_{k=1}^d(-1)^{\eps_{jk}}a_{jk}.$
\end{enumerate}
\end{lem}
\subsection{Block Frequency Sparse Functions}
\label{subsec:Alg_for_Special_Props_blsp}
Let us consider block frequency sparse functions (condition \ref{item:block} in Lemma \ref{lem:special_cases_B}) in more detail. In that case, the support sets $S_1,\dotsc,S_n$ are of the form
\[
 S_j=\left\{a_{j0},a_{j0}+1,\dotsc,a_{j0}+B-1\right\}, \quad j\in\{1,\dotsc,n\},
\] 
\begin{diss}
Input signals or functions of this type, often called \emph{multiband signals}, have already been considered in other publications, for example in \cite{eldar2010_theorytopractice,eldar2008_eff_sampl_sparsewideband,eldar2011_xampling_analog_digital_subnyquist,eldar2009_blind_multiband_signal_rec,fengbresler1996}. The special case where $n=1$, meaning that all possibly energetic frequencies occur in a single block, has also been investigated further in \cite{plonka_smallsupp} and \cite{bittens2016}.
\end{diss}
and we can improve the runtime of our algorithm even further. 
\begin{dfn}[$(n,B)$-block Sparsity]
 A $P(n,1,B)$-structured sparse function $f$ is called $(n,B)$\emph{-block sparse} if the support sets $S_1,\dotsc,S_n$ are generated by the monic linear polynomials
 \[
  P_j(x):=x+a_{j}, \quad j\in\{1,\dotsc,n\}.
 \]
\begin{diss}
 The frequency subsets are of the form $S_j=\{\omega_j,\omega_j+1,\dotsc,\omega_j+B-1\}$.
\end{diss}
\end{dfn}
For block sparse functions we can extend the definition of good hashing primes to integers, because we do not require the multiplicative invertibility of all nonzero elements anymore.
\begin{dfn}\label{dfn:wellhashed_int}
 Let $f$ be $(n,B)$-block sparse with support set $S=\bigcup_{j=1}^nS_j$ generated by the polynomials $P_1,\dotsc,P_n$. An integer $u>B$ \emph{hashes} a support set $S_j$ \emph{well} if
 \[
  |\{\omega \mods u:\omega\in S_j\}|=B \quad \forall j\in\{1,\dotsc,n\}.
 \] 
\end{dfn}
\begin{rem}
 For an $(n,B)$-block sparse function $f$ any integer $u>B$ hashes every support set $S_j$ well, since it consists of $B$ consecutive frequencies. Thus for every residue $\nu$ modulo $u$ the restriction of $S$ to the frequencies congruent to $\nu$ is at most  $n$-sparse,
 \[
  |\{\omega\equiv\nu\mods u\colon \omega\in S\}|\leq n \quad \forall \nu\in\{0,\dotsc,u-1\}.
 \]
\end{rem}
If $f$ is $(n,B)$-block sparse, we can choose the hashing integer $u$ to be the smallest power of 2 greater than the block length $B$. Then $u=\mathcal{O}(B)$, which allows us to give better runtime estimates. Additionally, computing DFTs of length $s\cdot t\cdot u$, where $s$ and $t$ are small primes and $u$ is a power of 2, is faster than if $u$ were a prime of the same size.

In Algorithm \ref{alg:blocksparse} we give the pseudocode for Algorithm \ref{alg:fourierapprox} in the special case of $(n,B)$-block sparse functions.
\begin{algorithm}
\caption{Fourier Approximation for $(n,B)$-block Sparse Functions}
\label{alg:blocksparse}
\begin{algorithmic}[1]
\renewcommand{\algorithmicrequire}{\textbf{Input:}}
\renewcommand{\algorithmicensure}{\textbf{Output:}}
\Require Function $f+\eta$, where $f$ is $(n,B)$-block sparse, bandwidth $N$.
\Ensure $R,x_R$, where $R$ contains the $nB$ frequencies $\omega$ with greatest magnitude coefficient estimates $x_R(\omega)$.
\State Set $u=2^\alpha$ with $\alpha=\lfloor\log_2B\rfloor+1$ and find $L$ and the smallest odd primes $t_1,\dotsc,t_L$ such that $\prod_{l=1}^{L-1}t_l<\frac{N}{un}\leq\prod_{l=1}^Lt_l$. Set $t_0=1$.
\State Let $s_1>\max(n,t_L)$ be prime, $K=2n\lfloor\log_{s_1}\frac{N}{u}\rfloor+1$ and $s_1<\dotsb<s_K$ be primes.
% such that $s_1,\dotsc,s_K,t_1,\dotsc,t_L$ are pairwise relatively prime.
\State Initialize $R=\emptyset$, $\x_R=\mathbf{0}_{N}$
\For{$k$ from 1 to $K$} \label{line:dftstart2} \Comment{computation of $\GGG$ and $\EEE$}
\For{$l$ from 0 to $L$}
\State $\boldsymbol A_{s_kt_lu}\gets\left(f\left(\frac{2\pi j}{s_kt_lu}\right)\right)_{j=0}^{s_kt_lu-1}$
\State $\widehat{\boldsymbol A_{s_kt_lu}}\gets\textbf{DFT}[\boldsymbol A_{s_kt_lu}]$
\EndFor
\EndFor \label{line:dftend2}

\Statex \begin{center} \textsc{Identification of the Energetic Frequencies} \end{center}
\For{$\nu$ from 0 to $u-1$} \label{line:countstart}
\For{$k$ from 1 to $K$}
\State $(1,v,w)\gets\texttt{extended\_gcd}(s_k,u)$ \Comment{i.e., $1=v\cdot s_k+w\cdot u$} \label{line:euclidean}
\For{$h$ from 0 to $s_k-1$}\label{line:bthlargeststart}
\State $a_{0}\gets ((h-\nu)w\mods s_k)\cdot u+\nu$ \Comment{residue modulo $s_ku$}
\label{line:bthlargest}
\For{$l$ from 1 to $L$} \label{line:minimumstart}
\State $b_{\text{min}}\gets\underset{b\in\{0,\dotsc,{t}_l-1\}}{\argmin}\left|\widehat{A_{s_ku}}(a_{0})
-\widehat{A_{s_kt_lu}}(a_{0}+b\cdot s_ku)\right|$ \label{line:argmin}
\State $a_{l}\gets(a_{0}+b_{\text{min}}s_ku)\mods t_l$ \label{line:residues}
\EndFor \label{line:minimumend}
\State Reconstruct $\omega$ by $\omega\equiv a_{0}\mods s_ku$, $\omega\equiv a_{1}\mods t_1,\dotsc,\omega\equiv a_{L}\mods t_L$.\label{line:reconstruction}
\EndFor\label{line:bthlargestend}
\EndFor \label{line:identification_end}

\Statex \begin{center} \textsc{Fourier Coefficient Estimation} \end{center}
\For{\textbf{each} $\omega\equiv\nu\mods u$ reconstructed more than $\frac{K}{2}$ times}\label{line:energetic}
%\For{$\omega$ from $\tilde\omega-B+1$ to $\tilde\omega+B-1$} \Comment{new block found} \label{line:estimate_inner_start}
\State $\re\left(x_{\omega}\right)
\gets\text{median}\left\{\re\left(\widehat{A_{s_kt_Lu}}(\omega\mods s_kt_Lu)\right):k\in\{1,\dotsc,K\}\right\}$ \label{line:re}
\State $\im\left(x_{\omega}\right)
\gets\text{median}\left\{\im\left(\widehat{A_{s_kt_Lu}}(\omega\mods s_kt_Lu)\right):k\in\{1,\dotsc,K\}\right\}$ \label{line:im}
\EndFor
\State Sort the coefficients by magnitude s.t. $|x_{\omega_1}|\geq|x_{\omega_2}|\geq\dotsb$.
\State $R^{(1,\nu)}=\{\omega_1,\omega_2,\dotsc,\omega_{2n}\}$ 
\EndFor
\State Sort the coefficients in $\bigcup_{\nu=0}^{u-1}R^{(1,\nu)}$ by magnitude s.t. $|x_{\omega_1}|\geq|x_{\omega_2}|\geq\dotsb$.
\State Output: $R=\{\omega_1,\omega_2,\dotsc,\omega_{nB}\}, x_R$ \label{line:R_m_sparse}
\end{algorithmic}
\end{algorithm}

The function \texttt{extended\_gcd} in line \ref{line:euclidean} denotes the extended Euclidean algorithm, which finds the greatest common divisor $g$ of two integers $a$ and $b$, as well as two integers $v$ and $w$ such that Bézout's identity 
\[
 g=\gcd(a,b)=v\cdot a+w\cdot b
\]
is satisfied. By definition of $u$ and $s_k$ we always have $g=1$ in line \ref{line:euclidean}.

\begin{cor}\label{cor:multiblock}
 Let $N\in\NN$ and $f\colon[0,2\pi]\rightarrow\CC$ be $(n,B)$-block sparse with noise $\eta$ such that $\cc(\eta)\in\ell^1$ and $\|\cc(\eta)\|_\infty\leq\eps$. Set $u:=2^\alpha$, where $\alpha:=\left\lfloor\log_2 B\right\rfloor+1$, $M=1$ and the $s_k$ and $t_l$ as in Theorem \ref{thm:main}. If $u> s_K$, the runtime of Algorithm \ref{alg:fourierapprox} is given by 
 \[
  \mathcal{O}\left(\frac{B\log B\cdot n^2\log^2\frac{N}{Bn}\log^2\frac{N}{B}\log\left(n\log\frac{N}{B}\right)}{\log^2n\log\log\frac{N}{Bn}}\right), 
 \]
 and otherwise, if $u<s_K$, by
 \[
  \mathcal{O}\left(\frac{Bn^2\cdot\log^2\frac{N}{Bn}\log^2\frac{N}{B}\log^2\left(n\log\frac{N}{B}\right)}{\log^2n\log\log\frac{N}{Bn}}\right).
 \]
 In both cases the algorithm has a sampling complexity of
 \[
  \mathcal{O}\left(\frac{Bn^2\cdot\log^2\frac{N}{Bn}\log^2\frac{N}{B}\log\left(n\log\frac{N}{B}\right)}{\log^2n\log\log\frac{N}{Bn}}\right).
 \]
\end{cor}
\begin{proof}
Choosing $u$ as a power of 2 implies that we now have to slightly modify the $t_l$ and $s_k$. Similar to the choice of the primes in Remark \ref{rem:estimates}, we take the smallest $L$ odd primes such that their product is greater than or equal to $\frac{N}{un}$,
\[
 \prod_{l=1}^{L-1}t_l<\frac{N}{un}\leq\prod_{l=1}^Lt_l, \quad t_1:=3.
\]
This means that $t_l=p_{l+1}$. Let $s_1$ be the smallest prime that is greater than $n$ and $t_L$,
 \[
  s_1:=p_x>\max\{n,t_L\}\geq p_{x-1}.
 \]
In this setting we can use the minimal $K$,
\[
 K=8n\left\lfloor\log_{s_1}\frac{N}{u}\right\rfloor+1.
\]
The remaining $s_k$ can be set as $s_k:=p_{x-1+k}$ for $k\in\{1,\dotsc,K\}$. Then the set 
$\{t_1,\dotsc,t_L,s_1,\dotsc,s_K,u\}$ is pairwise relatively prime, $u>B$ and 
\[
 \prod_{l=1}^Lt_l\geq\frac{N}{s_1u_1},
\]
so the CRT can be applied. Since we chose $t_1=3$, the prime $t_L$ in this case is at most the smallest prime greater than the $t_L$ from Remark \ref{rem:estimates} for $d=1$, and thus we still have that 
\[
 t_L=\mathcal{O}\left(\log\frac{N}{un}\right)  \quad\text{and}\quad 
 s_K=\mathcal{O}\left(n\log_n\frac{N}{u}\log\left(n\log\frac{N}{u}\right)\right).
\]
Using additionally that $u=\mathcal{O}(B)$, the runtime of Algorithm \ref{alg:fourierapprox} for $u> s_K$ is given by 
\begin{diss}
\begin{align*}
 &\mathcal{O}\left(\sum_{k=1}^K\sum_{l=0}^Ls_kt_lu\log(s_kt_lu)\right)=\mathcal{O}\left(\sum_{k=1}^K\sum_{l=0}^Ls_kt_lu\log u\right) \\
 =&\mathcal{O}\left(u\log u\cdot\frac{s_K^2}{\log s_K}\cdot \frac{t_L^2}{\log t_L}\right) \\
 =&\mathcal{O}\left(B\log B\cdot\frac{n^2\log^2_n\frac{N}{u}\log^2\left(n\log \frac{N}{u}\right)}
 {\log\left(n\log_n\frac{N}{u}\log\left(n\log\frac{N}{u}\right)\right)}\cdot\frac{\log^2\frac{N}{u}}{\log\log\frac{N}{u}}\right) \\
 =&\mathcal{O}\left(\frac{B\log B\cdot n^2\log^2\frac{N}{Bn}\log^2\frac{N}{B}\log\left(n\log\frac{N}{B}\right)}{\log^2n\log\log\frac{N}{Bn}}\right). 
\end{align*}
\end{diss}
\begin{paper}
\begin{align*}
 &\mathcal{O}\left(\sum_{k=1}^K\sum_{l=0}^Ls_kt_lu\log(s_kt_lu)\right)
 =\mathcal{O}\left(u\log u\cdot\frac{s_K^2}{\log s_K}\cdot \frac{t_L^2}{\log t_L}\right) \\
 =&\mathcal{O}\left(\frac{B\log B\cdot n^2\log^2\frac{N}{Bn}\log^2\frac{N}{B}\log\left(n\log\frac{N}{B}\right)}{\log^2n\log\log\frac{N}{Bn}}\right). 
\end{align*}
\end{paper}
If $u<s_K$, we obtain a runtime of
\begin{diss}
\begin{align*}
 &\mathcal{O}\left(\sum_{k=1}^K\sum_{l=0}^Ls_kt_lu\log(s_kt_lu)\right)=\mathcal{O}\left(u\cdot\sum_{k=1}^K\sum_{l=0}^Lt_ls_k\log s_k\right) \\
 =&\mathcal{O}\left(u\cdot s_K^2\cdot \frac{t_L^2}{\log t_L}\right) \\
 =&\mathcal{O}\left(\frac{Bn^2\cdot\log^2\frac{N}{Bn}\log^2\frac{N}{B}\log^2\left(n\log\frac{N}{B}\right)}{\log^2n\log\log\frac{N}{Bn}}\right).
\end{align*}
\end{diss}
\begin{paper}
\begin{align*}
 &\mathcal{O}\left(\sum_{k=1}^K\sum_{l=0}^Ls_kt_lu\log(s_kt_lu)\right)=\mathcal{O}\left(u\cdot s_K^2\cdot \frac{t_L^2}{\log t_L}\right) \\
 =&\mathcal{O}\left(\frac{Bn^2\cdot\log^2\frac{N}{Bn}\log^2\frac{N}{B}\log^2\left(n\log\frac{N}{B}\right)}{\log^2n\log\log\frac{N}{Bn}}\right).
\end{align*}
\end{paper}

In both cases the number of required samples of $f+\eta$ is
\begin{align*}
 &\mathcal{O}\left(\sum_{k=1}^K\sum_{l=0}^Ls_kt_lu\right)=\mathcal{O}\left(u\cdot \frac{s_K^2}{\log s_K}\cdot\frac{t_L^2}{\log t_L}\right) \\
 =&\mathcal{O}\left(\frac{Bn^2\cdot\log^2\frac{N}{Bn}\log^2\frac{N}{B}\log\left(n\log\frac{N}{B}\right)}{\log^2n\log\log\frac{N}{Bn}}\right).
\end{align*}
\end{proof}

If $f+n$ is bandlimited, we can prove the 1-norm error bound in Theorem \ref{thm:main_intro_blockSparse} in \S\ref{ch:intro_results}.
\begin{cor}\label{cor:bandlimited_block}
 Let $N\in\NN$ and $f\colon[0,2\pi]\rightarrow\CC$ be $(n,B)$-block sparse with noise $\eta$ such that $\eta\in\ell^1$, $\|\cc(\eta)\|_\infty\leq\eps$ and $f$ and $f+\eta$ are bandlimited to $\intvl$. Setting $u:=u_1:=2^\alpha$, where $\alpha:=\left\lfloor\log_2 B\right\rfloor+1$, $M=1$ and the $s_k$ and $t_l$ as in Theorem \ref{thm:main}, the output $(R,\x_R)$ of Algorithm \ref{alg:fourierapprox} satisfies
 \[
   \|\cc(N)-\x_R\|_1\leq 4\left\|\cc(N)-\cc\opt_{Bn}(N)\right\|_1+2Bn\eps.
 \] 
\end{cor}
\begin{proof}
As we do not have to take medians over the estimates obtained for the different hashing primes, we can consider the following inequality,
 \begin{align*}
  &\|\cc(N)-\x_R\|_1=\sum_{\omega=\intu}^{\into}|c_\omega-x_\omega|=\sum_{\nu=0}^{u-1}\sum_{\substack{\omega=\intu \\ \omega\equiv\nu\mods u}}^{\into}|c_\omega-x_\omega| \\
  =&\sum_{\nu=0}^{u-1}\biggl(\sum_{\substack{\omega\in R^{(1,\nu)} \\ \omega\equiv\nu\mods u}} |c_\omega-x_\omega|
  +\sum_{\substack{\omega\notin R^{(1,\nu)}\\ \omega\equiv\nu\mods u}} |c_\omega| \biggr) \\
  =&\sum_{\nu=0}^{u-1}\biggl(\sum_{\omega\in R^{(1,\nu)}}|c_\omega-x_\omega|+\sum_{\omega\notin R^{(1,\nu),\text{opt}}_n}|c_\omega| +\sum_{\omega\in R^{(1,\nu),\text{opt}}_n\backslash R^{(1,\nu)}}|c_\omega| -\sum_{\omega\in R^{(1,\nu)}\backslash R^{(1,\nu),\text{opt}}_n}|c_\omega|\biggr).
 \end{align*}
 By (\ref{eq:someapprox}) the $2n$ elements of $R^{(1,\nu)}$ satisfy $|c_\omega-x_\omega|\leq\sqrt{2}\delta^{(1,\nu)}$. From the proof of Lemma \ref{lem:approx} it follows that $|c_\omega|\leq\eps+2\sqrt{2}\delta^{(1,\nu)}$ for all $\omega\in R^{(1,\nu),\text{opt}}_n\backslash R^{(1,\nu)}$, because otherwise $\omega$ would be added to $R^{(1,\nu)}$. Recall the definition of $\delta^{(1,\nu)}$,
 \[
  \delta^{(1,\nu)}:=\frac{1}{2n}\left\|\cc(N,u,\nu)-\cc_{2n}\opt(N,u,\nu)\right\|_1+\underbrace{\left\|\cc(N,\ZZ,u,\nu)-\cc(u,\nu)\right\|_1}_{=0},
 \]
 since $f+\eta$ is bandlimited. Combining these considerations we find that
 \begin{align*}
  &\|\cc(N)-\x_R\|_1
  \leq\sum_{\nu=0}^{u-1}\left(2\sqrt{2}n\delta^{(1,\nu)}+\left\|\cc(N,u,\nu)-\cc_n\opt(N,u,\nu)\right\|_1+n\left(\eps+2\sqrt{2}\delta^{(1,\nu)}\right)\right) \\
  =&\sum_{\nu=0}^{u-1}\left(\left\|\cc(N,u,\nu)-\cc_n\opt(N,u,\nu)\right\|_1 +4\sqrt{2}n\left(\frac{1}{2n}\left\|\cc(N,u,\nu)-\cc_{2n}\opt(N,u,\nu)\right\|_1\right)\right) \\ 
  +&nu\eps.
 \end{align*}
 Because $f$ is $(n,B)$-block sparse, every restriction $\cc(N,u,\nu)$ of $\cc(f+\eta)$ to the frequencies that are congruent to $\nu$ modulo $u$ is $n$-sparse, so we can use the same idea as in the proof of Corollary \ref{cor:bandlimited_general} to obtain
 \begin{align*}
  \|\cc(N)-\x_R\|_1
  \leq& \left\|\cc(N)-\cc_{Bn}\opt(N)\right\|_1+2\sqrt{2}\cdot\left\|\cc(N)-\cc\opt_{2Bn}(N)\right\|_1 +nu\eps \\
  \leq& 4\cdot\left\|\cc(N)-\cc_{Bn}\opt(N)\right\|_1+2Bn\eps.
 \end{align*}
\end{proof}

\section{Numerical Evaluation}
\label{sec:Numerical_Eval}

In this section we evaluate the performance of two different variants of Algorithm \ref{alg:fourierapprox} including $(i)$ the deterministic variant for block sparse functions described in \S\ref{subsec:Alg_for_Special_Props_blsp} (referred to as the \textbf{F}ourier \textbf{A}lgorithm for \textbf{S}tructured sparsi\textbf{T}y (FAST) below), and $(ii)$ a randomized implementation of Algorithm \ref{alg:fourierapprox} which only utilizes a small random subset of the $M$ hashing primes used by FAST for each choice of its parameters (referred to as the \textbf{F}ourier \textbf{A}lgorithm for \textbf{S}tructured sparsi\textbf{T}y with \textbf{R}andomization (FASTR) below).  Both of these C++ implementations are publicly available.\footnote{\url{http://na.math.uni-goettingen.de/index.php?section=gruppe&subsection=software}.}  We also compare these implementations' runtime and robustness characteristics with GFFT,\footnote{Also available at \url{http://na.math.uni-goettingen.de/index.php?section=gruppe&subsection=software}.} FFTW 3.3.4,\footnote{\url{http://www.fftw.org/}} and sFFT 2.0.\footnote{\url{https://groups.csail.mit.edu/netmit/sFFT/}}

Note that FAST and FASTR are both designed to approximate functions that are $(n,B)-$block sparse in Fourier space. This means that both FAST and FASTR take upper bounds on the number of blocks, $n$, and length of each block, $B$, present in the spectrum of the functions they aim to recover as parameters.  In contrast, both GFFT (a deterministic sparse Fourier transform \cite{segal2013improved}) and SFFT 2.0 (a randomized noise robust sparse Fourier transform \cite{HIKP}) only require an upper bound on the effective sparsity, $s$, of the function's Fourier coefficients.  Herein $s$ is always set so that $s=Bn$ for these methods.  Finally, FFTW is a highly optimized and publicly available implementation of the traditional FFT algorithm which runs in $\mathcal{O}(N \log N)$-time for input vectors of length $N$.  All the FFTW results below were obtained using FFTW 3.3.4 with its FFTW\_MEASURE plan. 

For the runtime experiments below the trial signals were formed by choosing sets of frequencies with $(n,B)-$block sparsity uniformly at random from $\intvl$.  Each frequency in this set was then assigned a magnitude $1$ Fourier coefficient with a uniformly random phase. The remaining frequencies were all set to zero.  Every data point in a figure below corresponds to an average over 100 trial runs on 100 different trial signals of this kind.  For different $n$, $B$ and $N$, the parameters in each randomized algorithm (i.e. FASTR and sFFT 2.0) were chosen so that the probability of correctly recovering an $(n,B)$-block sparse function was at least 0.9 for each run.  Finally, all experiments were run on a Linux CentOS machine with 2.50GHz CPU and 16 GB RAM. 

\subsection{Runtime as Block Length $B$ Varies: $N=2^{26}$, $n=2$ and $n=3$}

\begin{figure}[!tbp]
  \centering
%   \resizebox{\textwidth}{!}{\input{foo.tex}}
  \subfloat[Runtime comparison for bandwidth $N = 2^{26}$ and $n = 2$ blocks.]{\resizebox{0.47\textwidth}{!}{\input{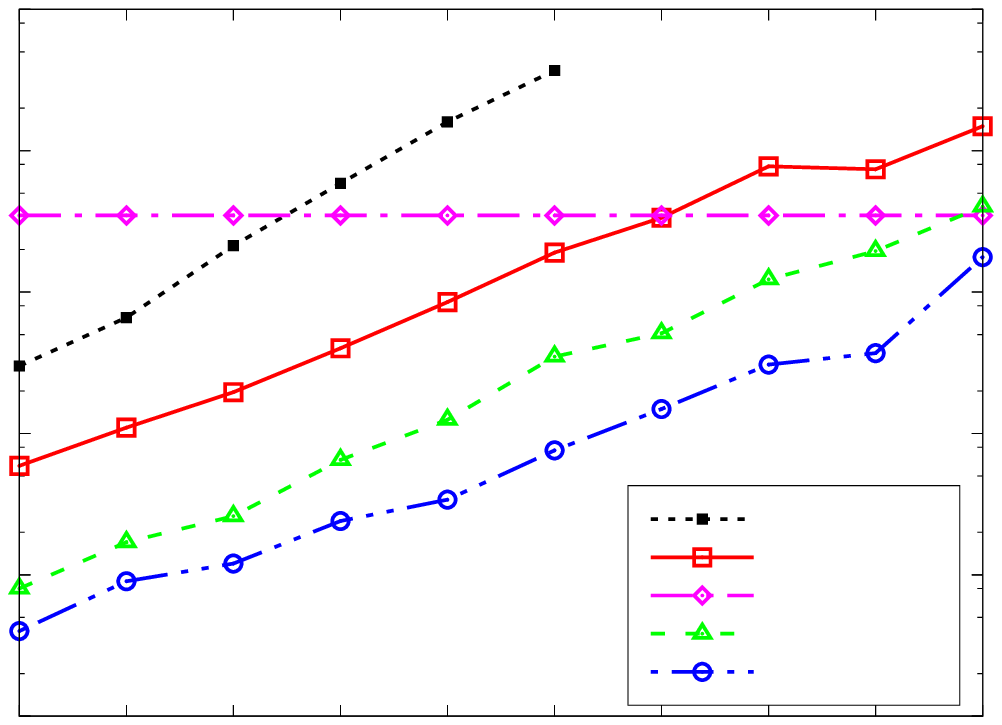}}\label{fig:n2}}
  \hfill
  \subfloat[Runtime comparison for bandwidth $N = 2^{26}$ and $n = 3$ blocks.]{\resizebox{0.47\textwidth}{!}{\input{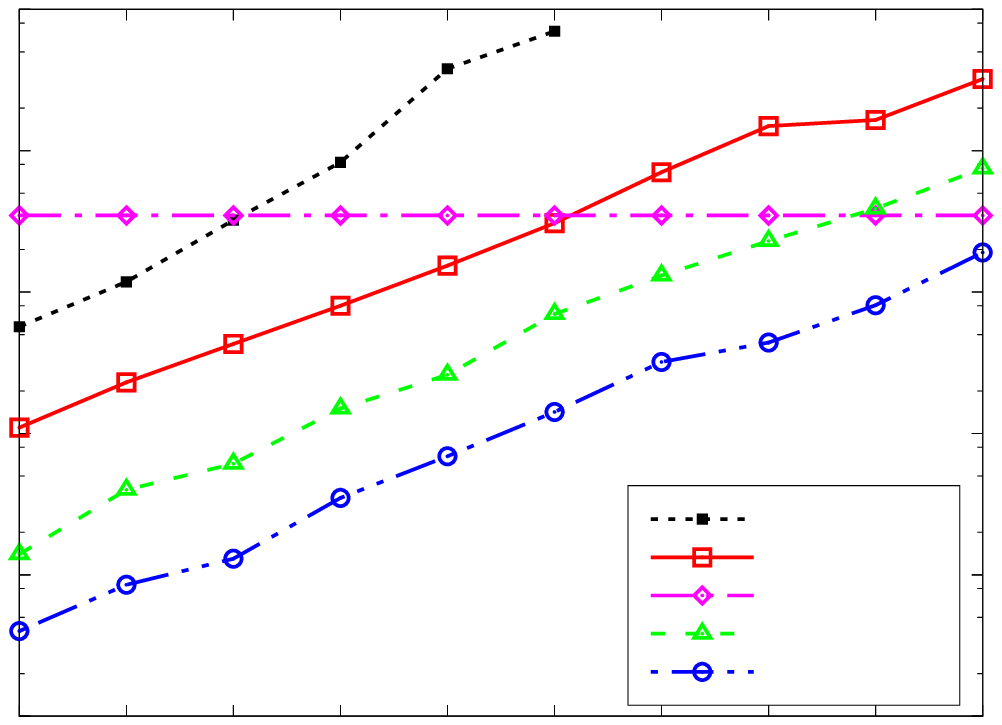}}\label{fig:n3}}
  \caption{Runtime plots for several algorithms and implementations of sparse Fourier transform for different $B$ settings}
\end{figure}

In Figure \ref{fig:n2} we fix the number of blocks to $n = 2$ and the bandwidth to $N = 2^{26}$, and then perform numerical experiments for 10 different block lengths $B$ = $2^{2}$, $2^{3}$, ..., $2^{11}$.  We then plot the runtime (averaged over 100 trial runs) for FAST, FASTR, GFFT, sFFT 2.0 and FFTW.  As expected, the runtime of FFTW is constant with increasing sparsity.  The runtimes of all the sparse Fourier transform algorithms other than GFFT are approximately linear in $B$, and they have similar slopes.  Figure \ref{fig:n2} demonstrates that allowing a small probability of incorrect recovery always lets the randomized algorithms (FASTR and sFFT 2.0) outperform the deterministic algorithms with respect to runtime. Among the deterministic algorithms, FAST is always faster than GFFT, and only becomes slower than FFTW when the value of $B$ is greater than 256.  The runtimes of both FASTR and sFFT 2.0 are still comparable with the one of FFTW when the block length $B$ is 2048. Comparing with sFFT 2.0, FASTR has better runtime performance on these block sparse functions, and is the only algorithm that is still faster than FFTW when $B=2048$.  In Figure \ref{fig:n3} we use the same settings of $N$ and $B$ as in the previous experiment and increase the number of blocks $n$ from 2 to 3.  With these settings the largest sparsity $s=Bn$ increases from 4048 ($2\cdot 2^{11}$) to 6144 ($3\cdot 2^{11}$).  The respective results for the methods are similar in this plot.

\subsection{Runtime as Number of Blocks $n$ Varies: $N = 2^{26}$ and $B = 32$}

\begin{figure}
  \centering
  \resizebox{0.6\textwidth}{!}{\input{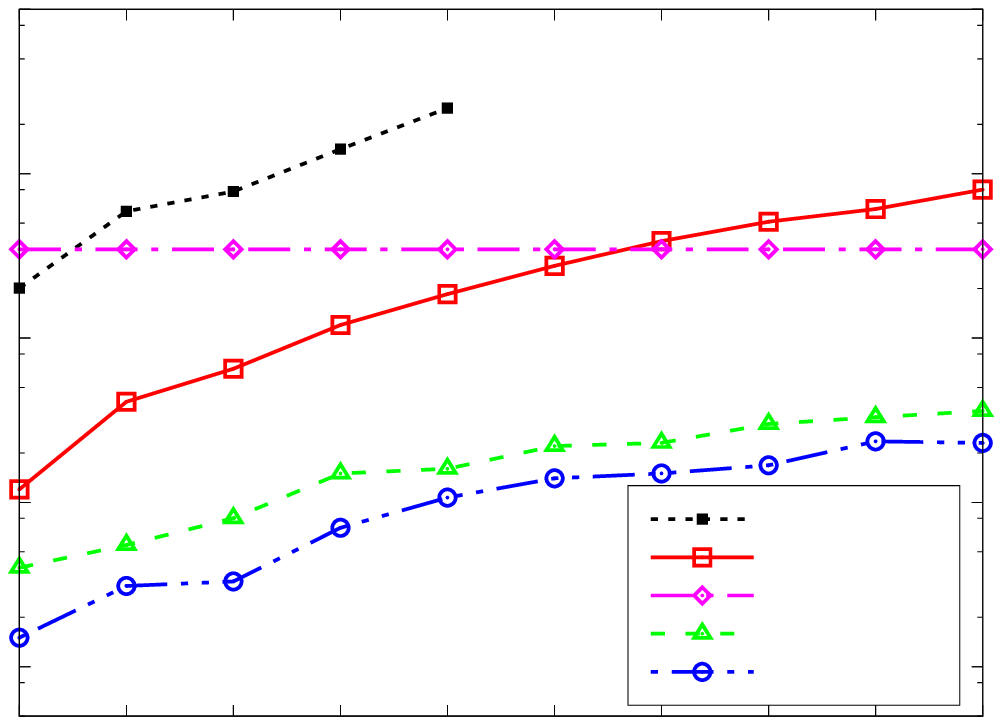}}
  \caption{Runtime comparison for bandwidth $N=2^{26}$ and block length $B=32$.}
  \label{fig:N26B32}
\end{figure}

In Figure \ref{fig:N26B32} we fix the bandwidth $N = 2^{26}$ and block length $B = 32$, then vary the number of blocks $n$ from 1 to 10.  Looking at  Figure \ref{fig:N26B32}, we can see that the deterministic sparse FFTs, GFFT and FAST, both have runtimes that increase more rapidly with $n$ than those of their randomized competitors.  Among the three deterministic algorithms, FAST has the best performance when the number of blocks is smaller than 6. Similar to the previous experiments, FFTW becomes the fastest deterministic algorithm when the sparsity $s = Bn$ gets large enough (greater than $224$ in this experiment). The two randomized algorithms are both faster than FFTW by an order of magnitude when the number of blocks is $10$. Similarly, FASTR is always faster than sFFT 2.0 for the examined value of $N$.

\subsection{Runtime as Signal Size $N$ Varies:  $n = 2$ and $B = 64$}

\begin{figure}
  \centering
    \resizebox{0.6\textwidth}{!}{\input{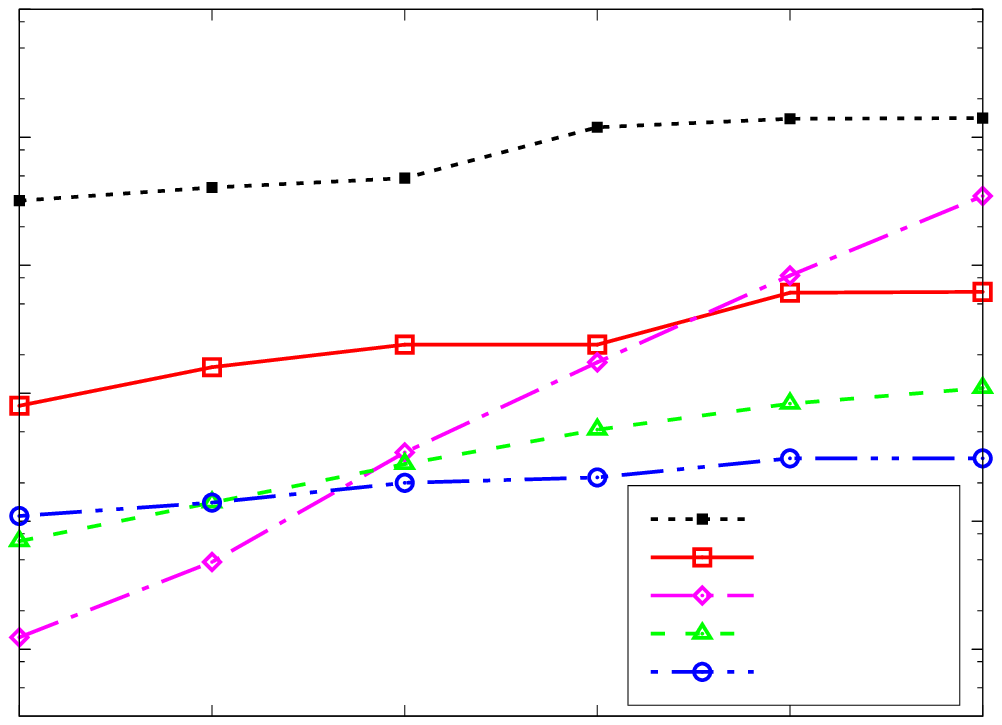}}
  \caption{Runtime comparison for $n=2$ blocks of length $B=64$.}
  \label{fig:n2B64}
\end{figure}

In Figure \ref{fig:n2B64} we fix the number of blocks $n = 2$ and block length $B=64$, then test the performance of the different algorithms with various bandwidths $N$. It can be seen in Figure \ref{fig:n2B64} that FFTW is the fastest deterministic algorithm for small bandwidth values. However, the runtime of FFTW becomes slower than the one of FAST when the bandwidth $N$ is greater than $2^{24}$. GFFT is the slowest deterministic algorithm for this sparsity level for all plotted $N$.  Comparing randomized SFT algorithms, FASTR always performs better than sFFT 2.0 when the bandwidth is greater than $2^{18}$. 

\subsection{Robustness to Noise}

To test the robustness of the methods to noise we add Gaussian noise to each of the signal samples utilized in each method and then measure the contamination of the recovered Fourier series coefficients for $(n,B)$-block sparse functions $f\colon  [0, 2 \pi] \rightarrow \mathbb{C}$ with bandwidth $N=2^{22}$, number of blocks $n=3$, and block length $B=2^{4}$.  More specifically, each method considered herein utilizes a set of samples from $f$ given by ${\f} = \left(  f(x_j) \right)^{m-1}_{j=0}$ for some $x_0, \dots, x_{m-1} \in [0, 2\pi)$ with $m \leq N$.  For the experiments in this section we instead provide each algorithm with noisy function evaluations of the form $\left(  f(x_j) + n_j\right)^{m-1}_{j=0}$,
where each $n_j \in \mathbb{C}$ is a complex Gaussian random variable with mean 0.  The $n_j$ are then rescaled so that the total additive noise $\mathbf{n} = \left( n_j \right)^{m-1}_{j=0}$ achieves the signal-to-noise ratios (SNRs) considered in Figure \ref{fig:robustness}.\footnote{The SNR is defined to be $\text{SNR}=20\log \left( \frac{\parallel {\f} \parallel_2}{\parallel \mathbf{n} \parallel_2} \right)$, where ${\f}$ and $\mathbf{n}$ are as given above.}

\begin{figure}
  \centering
    \resizebox{0.6\textwidth}{!}{\input{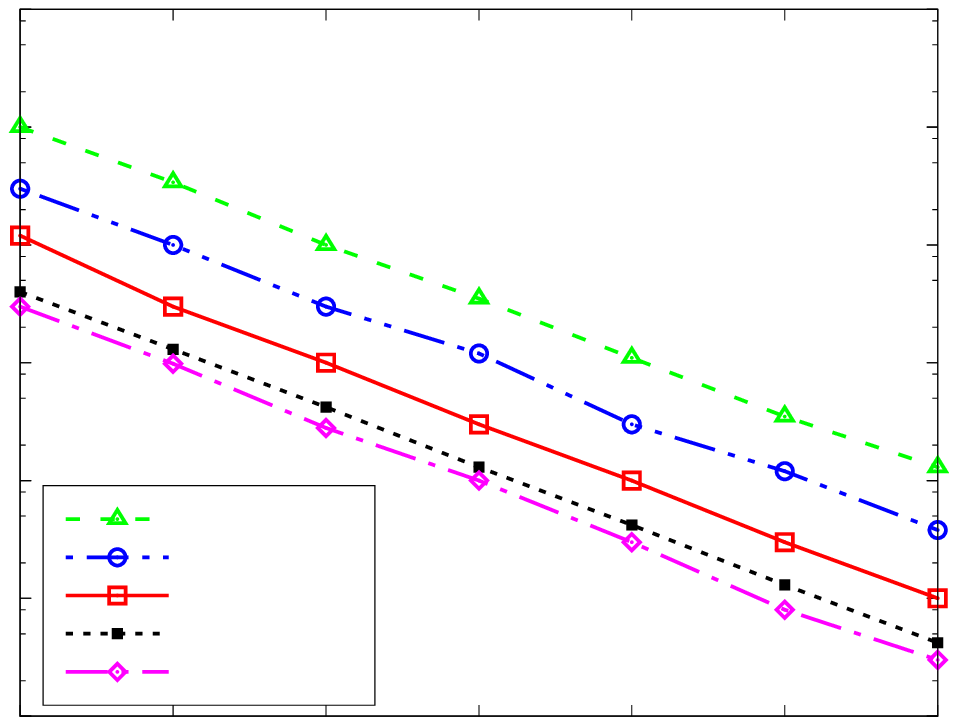}}
  \caption{Robustness to noise for bandwidth $N=2^{22}$ and $n=3$ blocks of length $B=2^4$.}
  \label{fig:robustness}
\end{figure}

Recall that the two randomized algorithms compared herein (SFT 2.0 and FASTR) are both tuned to guarantee exact recover of block sparse functions with probability at least 0.9 in all experiments.  For our noise robustness experiments this ensures that the correct frequency support, $S$, is found for at least 90 of the 100 trial signals used to generate each point plotted in Figure \ref{fig:robustness}.  All the other (deterministic) methods always find this correct support for all noise levels considered herein after sorting their output Fourier coefficient estimates by magnitude.  Figure \ref{fig:robustness} plots the average $\ell^1$-error over the true Fourier coefficients for frequencies in the correct frequency support $S$ of each trial signal, averaged over the at least 90 trial runs at each point for which each sparse Fourier transform correctly identified $S$.  More specifically, it graphs 
\[
\frac{1}{Bn}\sum_{\omega \in S} \big| c_\omega - x_\omega \big |,
\]
where $c_\omega$ are the true Fourier coefficients for frequencies $\omega \in S$, and $x_\omega$ are their recovered approximations, averaged over the at least 90 trial signals where each method correctly identified $S$.

Looking at Figure \ref{fig:robustness} one can see that all of the Fourier transform algorithms in our experiments are robust to noise. Overall, however, the deterministic algorithms (FAST, GFFT and FFTW) are more robust than randomized algorithms (FASTR and sFFT 2.0).  As expected, FFTW is the most robust algorithm in this experiment, followed closely by GFFT.  For the randomized algorithms, FASTR is more robust than sFFT 2.0.

\section{Conclusion}
\label{sec:Conclusion}

In this paper we developed the fastest known deterministic SFT method for the recovery of polynomially structured sparse input functions.  However, there are still some remaining avenues for future research. To begin with one could try to find other types of structured sparsity that also guarantee an upper bound on the sparsity of the frequency restrictions for all possible residues. Considering a structure generated by polynomials was merely the most obvious choice, as polynomials naturally agree well with hashing modulo prime numbers and therefore interact well with the the number theoretic constructions used herein.
One could also investigate whether utilizing structured sparsity might actually improve the runtimes of existing randomized SFT algorithms for unstructured sparsity. %%Another important question is whether randomizing our method for structured frequency sparse functions can decrease its runtime at the cost of a small probability of failure. As there exists a randomized version of the underlying Algorithm 3 in \cite{Iw-arxiv}, it should be easy to randomize Algorithm \ref{alg:fourierapprox}. The fastest existing randomized algorithms for unstructured sparse functions have a runtime that scales linearly in the sparsity; hence it is unlikely that randomizing our method can improve on that. However, due to the fact that it respects the sparsity structure of the function, it is still conceivable that a randomized version of Algorithm \ref{alg:fourierapprox} might be faster than the known randomized methods for unstructured sparsity.

It would also be interesting to know whether the results presented herein can be transferred to the non-periodic, continuous case, i.e., to  sparse functions defined on the whole real line. Results in \cite{BCGLS} about porting randomized SFT algorithms to the continuous setting suggest that this should be possible for a randomized version of Algorithm \ref{alg:fourierapprox}.

\section*{Acknowledgements}  

Sina Bittens was supported in part by the DFG in the framework of the GRK 2088.  Mark Iwen and Ruochuan Zhang were both supported in part by NSF DMS-1416752.  The authors would also like to thank both Felix Krahmer for introducing them at TUM in the summer of 2016, as well as Gerlind Plonka for her ongoing support, and particularly for her generosity in providing resources that aided in the writing of this paper.

\bibliographystyle{abbrv}
\bibliography{SFFTrefs.bib}
\newpage

% \thispagestyle{empty}
% $  $

%\listoftodos

\end{document}